\documentclass[11pt,letterpaper,reqno]{amsart}

\usepackage[a-2b,mathxmp]{pdfx}[2018/12/22]

\usepackage{geometry}
\usepackage{amssymb,thmtools}

\usepackage[shortlabels]{enumitem}

\usepackage[spacing=true,kerning=true,babel=true,tracking=true]{microtype}

\usepackage{wrapfig}
\usepackage{float}

\usepackage{hyperref}
\hypersetup{
	pdfstartview={XYZ null null 1.00}, 
	pdfpagemode=UseNone, 
	colorlinks,
	breaklinks, 
	linkcolor=blue,
	urlcolor=blue, 
	anchorcolor=blue,
	citecolor=blue
}

\usepackage[capitalise]{cleveref}
\crefformat{footnote}{#2\footnotemark[#1]#3}

\usepackage{comment}

\usepackage{xspace}
\usepackage{bbold}
\usepackage{graphicx}

\usepackage{etoolbox}

\makeatletter
\def\subsection{\@startsection{subsection}{2}%
  \z@{.8\linespacing\@plus.7\linespacing}{.5\linespacing}%
  {\normalfont\bfseries}}

\def\@seccntformat#1{%
  \protect\textup{%
    \protect\@secnumfont
    \expandafter\protect\csname format#1\endcsname 
    \csname the#1\endcsname
    \protect\@secnumpunct
  }%
}

\patchcmd{\@startsection}
  {\@afterindenttrue}
  {\@afterindentfalse}
  {}{}
 \makeatother

\makeatletter
\def\l@section{\@tocline{1}{5pt}{0pc}{}{}}
\renewcommand{\tocsection}[3]{%
	\indentlabel{\@ifnotempty{#2}{\makebox[20pt][l]{%
				\ignorespaces#1 #2.\hfill}}}\sc #3\dotfill}

\newdimen{\tocsubsecmarg}
\def\l@subsection{\@tocline{2}{3pt}{0pc}{\tocsubsecmarg}{}}
\renewcommand{\tocsubsection}[3]{%
	\indentlabel{\@ifnotempty{#2}{\makebox[30pt][l]{%
				\ignorespaces#1 #2.\hfill}}}#3\dotfill}
\makeatother
\let\oldtocsubsection=\tocsubsection
\renewcommand{\tocsubsection}[2]{\hspace{3em} \oldtocsubsection{#1}{#2}}

\numberwithin{equation}{section}

\theoremstyle{plain}

\newtheorem{lemma}[equation]{Lemma}

\crefname{prop}{Proposition}{Propositions}
\newtheorem{prop}[equation]{Proposition}

\newtheorem{theorem}[equation]{Theorem}

\crefname{obs}{Observation}{Observations}
\newtheorem{obs}[equation]{Observation}

\crefname{cor}{Corollary}{Corollaries}
\newtheorem{cor}[equation]{Corollary}

\crefname{fact}{Fact}{Facts}

\makeatletter
\def\@empty{}
\def\ifemptycredit#1{%
	\def\tmp{#1}%
	\ifx\tmp\@empty%
	\else%
	{~(#1)}%
	\fi%
}
\makeatother

\newenvironment{namedthm*}[2][]{
\medskip\par\noindent \textbf{#2}\ifemptycredit{#1}\textbf{.}\itshape\xspace
}{\medskip}

\theoremstyle{definition}

\crefname{defn}{Definition}{Definitions}
\newtheorem{defn}[equation]{Definition}

\crefname{example}{Example}{Examples}

\crefname{question}{Question}{Question}
\newtheorem{question}[equation]{Question}

\theoremstyle{remark}

\newtheorem*{acknowledge}{Acknowledgements}

\crefname{remark}{Remark}{Remarks}
\newtheorem{remark}[equation]{Remark}

\crefname{claim}{Claim}{Claims}
\newtheorem{claim}{Claim}
\newtheorem*{claim*}{Claim}

\crefname{assumption}{Assumption}{Assumptions}
\newtheorem{assumption}[equation]{Assumption}

\declaretheoremstyle[
spaceabove=\topsep, 
spacebelow=6pt,
headfont=\normalfont\itshape,
notefont=\normalfont, notebraces={(}{)},
bodyfont=\normalfont,
postheadspace=4pt,
qed=\mbox{\smaller[4]$\boxtimes$}
]{claimproofstyle}
\declaretheorem[name={Proof of Claim}, style=claimproofstyle, unnumbered]{pf}

\let\#\mathbb
\let\@\mathcal
\let\^\widehat

\let\dfn\textbf

\newcommand{\F}{\mathbb{F}}
\newcommand{\N}{\mathbb{N}}
\renewcommand{\P}{\mathbb{P}}
\newcommand{\Q}{\mathbb{Q}}
\newcommand{\R}{\mathbb{R}}
\newcommand{\Z}{\mathbb{Z}}

\newcommand{\1}{\mathbb{1}}

\newcommand{\m}{\mathfrak{m}}
\newcommand{\coc}{\rho}

\newcommand{\finS}{J}
\newcommand{\tree}{\tau}
\newcommand{\Symb}{S}
\newcommand{\shift}{\sigma}

\newcommand{\0}{\mathbb{\emptyset}}

\newcommand{\e}{\varepsilon}

\renewcommand{\phi}{\varphi}

\newcommand{\partialto}{\rightharpoonup}
\newcommand{\conc}{{^\smallfrown}}

\newcommand{\actson}{\curvearrowright}

\newcommand*{\defeq}{\mathrel{\vcenter{\baselineskip0.5ex \lineskiplimit0pt \hbox{\scriptsize.}\hbox{\scriptsize.}}}=}

\newcommand*{\defequiv}{\mathrel{\vcenter{\baselineskip0.5ex \lineskiplimit0pt \hbox{\scriptsize.}\hbox{\scriptsize.}}}\Leftrightarrow~}

\newcommand*{\defequivlong}{\mathrel{\vcenter{\baselineskip0.5ex \lineskiplimit0pt \hbox{\scriptsize.}\hbox{\scriptsize.}}}\Longleftrightarrow~}

\newcommand{\dom}{\mathrm{dom}}
\newcommand{\im}{\mathrm{im}}
\newcommand{\graph}{\mathrm{graph}}

\newcommand{\set}[1]{\left\{ #1 \right\}}
\newcommand{\gen}[1]{\langle #1 \rangle}
\newcommand{\norm}[1]{\left\| #1 \right\|}
\newcommand{\rest}[1]{\mathord{\downharpoonright_{#1}}}

\newcommand{\eqcomment}[1]{\Big[\text{#1}\Big] \hspace{6pt}}

\usepackage{color}

\definecolor{purple}{RGB}{116,0,159}

\definecolor{vert}{RGB}{7,126,26}

\newcommand{\Folner}{F{\o}lner\xspace}

\newcommand{\fsup}[1][f]{{#1}^\star}
\newcommand{\finf}[1][f]{{#1}_\star}

\newcommand{\finv}[1][f]{\overline{#1}}

\newcommand{\backop}[1][T]{P_{#1}}

\usepackage{palatino}

\title{A backward ergodic theorem along trees \\ and its consequences for free group actions}

\author{Anush Tserunyan}
\address[Anush Tserunyan]{Mathematics and Statistics, McGill University, Montréal, QC, Canada}
\email{anush.tserunyan@mcgill.ca}

\author{Jenna Zomback}
\address[Jenna Zomback]{Mathematics, Amherst College, Amherst, MA, USA}
\email{jzomback@amherst.edu}

\thanks{The authors' research was partially supported by NSF Grant DMS-1855648 and the first author was also partially supported by NSERC Discovery Grant RGPIN-2020-07120.}



\date{\today}

\begin{document}

\begin{abstract}
We prove a new pointwise ergodic theorem for probability-measure-preserving (pmp) actions of free groups, where the ergodic averages are taken over arbitrary finite subtrees of the standard Cayley graph rooted at the identity. This result is a significant strengthening of a theorem of Grigorchuk (1987) and Nevo and Stein (1994), and a version of it was conjectured by Bufetov in 2002. 

Our theorem for free groups arises from a new -- backward -- ergodic theorem for a countable-to-one pmp transformation, where the averages are taken over arbitrary trees of finite height in the backward orbit of the point (i.e.\ trees of possible pasts). We also discuss other applications of this backward theorem, in particular to the shift map with Markov measures, which yields a pointwise ergodic theorem along trees for the boundary actions of free groups.
\end{abstract}

\maketitle

\vspace{-.7cm}

\tableofcontents

\section{Introduction}

The classical pointwise ergodic theorem, whose first instance dates back to Birkhoff \cite{Bir}, states that for any \dfn{measure-preserving} (i.e., $T_* \mu = \mu$) transformation $T : X \to X$ on a standard probability space $(X,\mu)$ 
%
and $f \in L^1(X,\mu)$, $\lim_{n \to \infty} \frac{1}{n+1} \sum_{k=0}^{n} f(T^k x) = \finv$ for a.e.\ $x \in X$, where $\finv$ is the conditional expectation of $f$ with respect to the $\sigma$-algebra of $T$-invariant Borel sets.

More generally, pointwise ergodic theorems have been proven for probability-measure-preserving (pmp) actions of countable (semi)groups perhaps with weighted ergodic averages (see \cref{subsec:intro:history} for a survey). The first such results for nonamenable groups are due to Grigorchuk \cite{Gri87,Gri99,Gri00}, Nevo \cite{Nevo}, and Nevo and Stein \cite{NS}, which apply to pmp actions of finitely generated free groups $\F_r$ (see \cref{forward_erg:triangles:simple}). 
These results state the convergence (to the conditional expectation) of weighted ergodic averages taken over the balls in the standard Cayley graph of $\F_r$, where the weights are uniform over each sphere.
This was later generalized by Bufetov \cite[Theorem 1]{Buf2} to finitely generated free semigroups and to a larger class of weight assignments.

\begin{wrapfigure}[12]{r}[1cm]{.3\textwidth}
\centering
    \includegraphics[height=4cm]{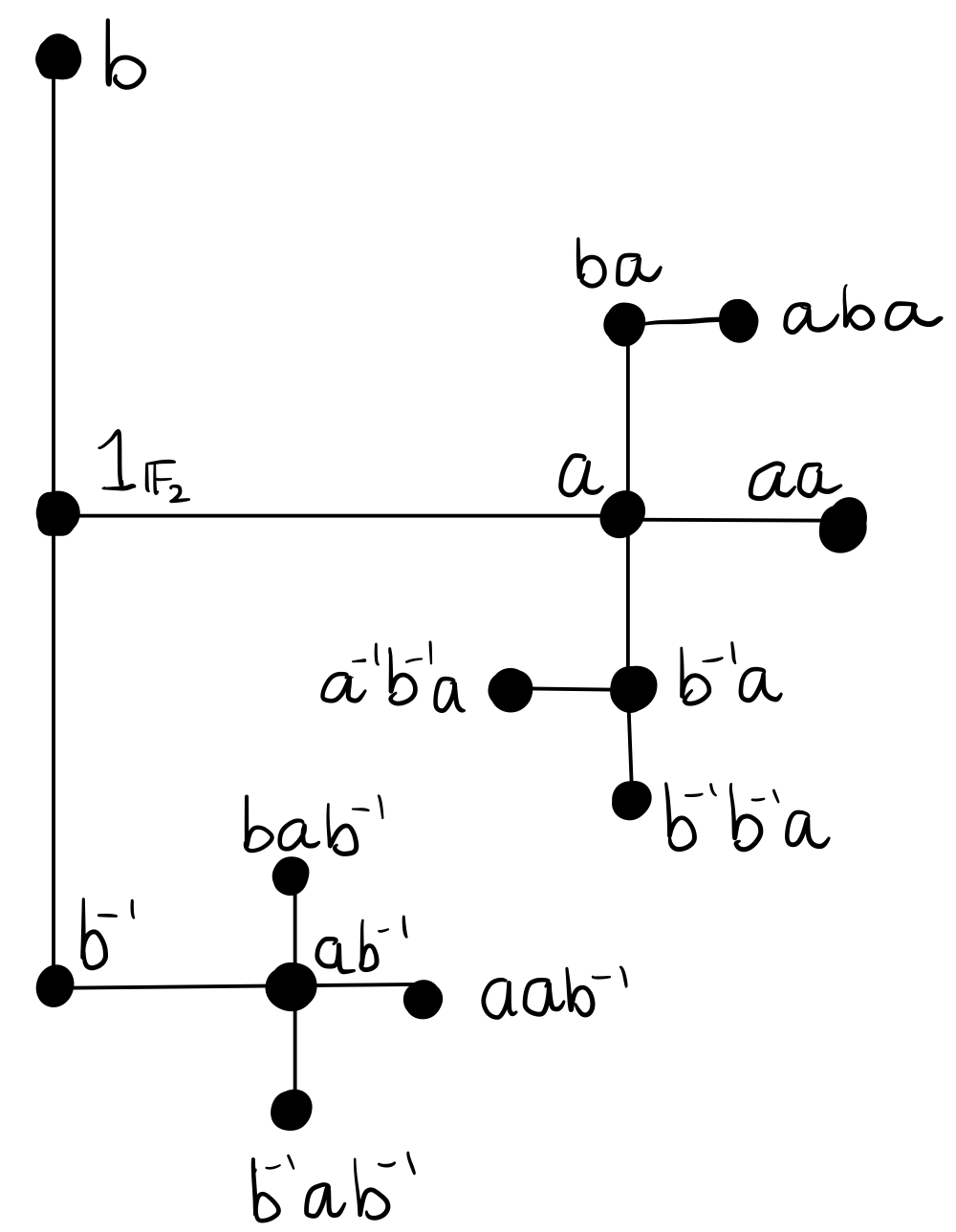}
    \caption{Example of a subtree $\tree$ \\ for $\F_2=\langle a,b \rangle$}
    \label{fig:tree_groups}
\end{wrapfigure}

\subsection{Results}\label{subsec:intro:results}

We prove a pointwise ergodic theorem for pmp actions of free groups where the ergodic averages are taken over \textbf{arbitrary finite subtrees} of the standard (left) Cayley graph; see \cref{fig:tree_groups} for an example of such a tree.
A version of this was conjectured by Bufetov in 2002\footnote{Bufetov has kindly allowed us to state this, and he has also informed us that he asked closely related questions at ESI, Vienna in 2002.}, and it vastly strengthens the aforementioned theorem of Grigorchuk, Nevo, and Nevo--Stein. 

\begin{theorem}[Pointwise ergodic for pmp actions of free groups]\label{intro:forward_erg:trees} 
Let $\F_r$ be the free group on $2 \le r < \infty$ generators and let $\F_r \actson^\alpha (X,\mu)$ be a (not necessarily free) pmp action of $\F_r$. Let $\Symb_r$ be the standard symmetric set of generators of $\F_r$ and let $\m_u$ be the uniform\footnote{
\label{ftnt:nonbacktracking-rw}$\m$ is the nonbacktracking simple symmetric random walk on the standard (left) Cayley graph of $\F_r$, i.e.\ $\m_u(w) \defeq \frac{1}{2r (2r-1)^{\ell-1}}$ for each reduced word $w \in \F_r$ of length $\ell \ge 1$, and $\m_u(1_{\F_r}) \defeq 1$.
} Markov measure on $\F_r$. 
Then for every $f \in L^1(X,\mu)$, for a.e.\ $x \in X$, for all sequences $(\tree_n)$ of finite subtrees of the (left) Cayley graph of $\F_r$ containing the identity (see \cref{fig:tree_groups}) such that $\lim_n \m_u(\tree_n) = \infty$, we have
\[
\lim_{n \to \infty} \frac{1}{\m_u(\tree_n)}\sum_{w\in \tree_n} f(w\cdot x) \m_u(w) = \finv(x),
\]
where $\finv$ is the conditional expectation of $f$ with respect to the $\sigma$-algebra of $\alpha$-invariant Borel sets.
\end{theorem}

A more general version of this theorem for a wider class of Markov measures is stated later as \cref{forward_erg:trees}.
Taking $\tree_n$ to be the ball of radius $n$ in $\F_r$ gives the conclusion of \cite{Gri87,Gri99,Gri00,NS}.

We also prove a pointwise ergodic theorem for the (non-pmp) action of the free group $\F_r$ on $1 \le r \le \infty$ generators on its boundary $\partial \F_r$, where we identify $\partial \F_r$ with the space of infinite reduced words on the standard symmetric set of generators of $\F_r$.
We denote by $u \conc v$ the concatenation of the words $u,v \in \F_r$.

\begin{theorem}[Pointwise ergodic for boundary actions of free groups]\label{intro:forward_erg:boundary}
For $1 \le r \le \infty$, let $\Symb_r$ be the standard symmetric set of generators of the free group $\F_r$, and let $\F_r \actson^\beta \partial \F_r$ be the natural action of $\F_r$ on its boundary $\partial \F_r \subseteq \Symb_r^\N$. Let $\m$ be a stationary Markov measure on the set of finite words in $\Symb_r$ whose support is contained in $\F_r$ (the set of reduced words)\cref{ftnt:nonbacktracking-rw} and let $\P_\m$ denote the induced Markov (probability) measure on $\partial \F_r$. For every $f \in L^1(\partial\F_r,\P_\m)$, for a.e.\ $x \in \partial \F_r$, for all sequences $(\tree_n)$ of finite subtrees of the (left) Cayley graph of $\F_r$ containing the identity but not $x(0)^{-1}$ (see \cref{fig:tree_groups} for $x(0) = a$) such that $\lim_n \m_u(\tree_n) = \infty$, we have
\[
\lim_{n \to \infty} \frac{1}{\m(\tree_n \conc x(0))} \sum_{w \in \tree_n}f(w \cdot x)\m(w \conc x(0)) = \finv(x),
\]
where $\tree_n \conc x(0) \defeq \set{w \conc x(0) : w \in \tree_n}$ and $\finv$ is the conditional expectation of $f$ with respect to the $\sigma$-algebra of $\beta$-invariant Borel sets.
\end{theorem}

We restate this theorem later as \cref{forward_erg:boundary}, which also includes the corresponding version of the maximal ergodic theorem, as well as convergence in $L^p$ for a specific sequence of trees. 
Bowen and Nevo in \cite[Theorem 4.1]{BN1} also provide a pointwise ergodic theorem for the boundary action of the free group, more precisely, for the diagonal action of the free group on the product of its boundary and a pmp action, and their averages are taken over horospherical balls. 
We also prove an ergodic theorem for this diagonal action (see \cref{sec:pmp_actions}) but we sample the averages over subtrees of the Cayley graph of $\F_r$ as in \cref{intro:forward_erg:boundary}.


The main result underlying \cref{intro:forward_erg:trees} and \cref{intro:forward_erg:boundary} is a backward pointwise ergodic theorem for a pmp Borel transformation $T$ (\cref{intro:backward_erg:trees}, restated later as \cref{backward_erg:trees}), where the averages are taken along trees of possible pasts (in the direction of $T^{-1}$). Although $T$ is pmp, the induced orbit equivalence relation $E_T$ is not pmp, unless $T$ is one-to-one, so the averages are weighted by the Radon--Nikodym cocycle of $E_T$ with respect to the measure.

\begin{theorem}[Backward pointwise ergodic along trees]\label{intro:backward_erg:trees}
Let $T$ be an aperiodic\footnote{$T$ is called \dfn{aperiodic} if for all $x\in X$ and $n\in \N\setminus\set{0}$, $T^n(x)\neq x$.} countable-to-one measure-preserving transformation on a standard probability space $(X,\mu)$. Let $E_T$ denote the induced orbit equivalence relation and let $(x,y) \to \coc_x(y) : E_T \to \R^+$ be the Radon--Nikodym cocycle of $E_T$ with respect to $\mu$. For every $f \in L^1(X,\mu)$, for a.e.\ $x \in X$,
\[
\frac{1}{\coc_x(\tree_x)}\sum_{y \in \tree_x} f(y) \coc_x(y) \to \finv(x) \;\text{ as }\; \coc_x(\tree_x)\to \infty,
\]
where $\tree_x$ ranges over all (possibly infinite) subtrees of the graph of $T$ of finite height rooted at $x$, directed towards $x$ (see \cref{fig:tree}), and where $\coc_x(\tree_x)\defeq\sum_{y\in \tree_x}\coc_x(y)$ and $\finv$ is the conditional expectation of $f$ with respect to the $\sigma$-algebra of $T$-invariant Borel sets. 
\end{theorem}

\begin{figure}[H]
    \centering
    \includegraphics[height=4cm]{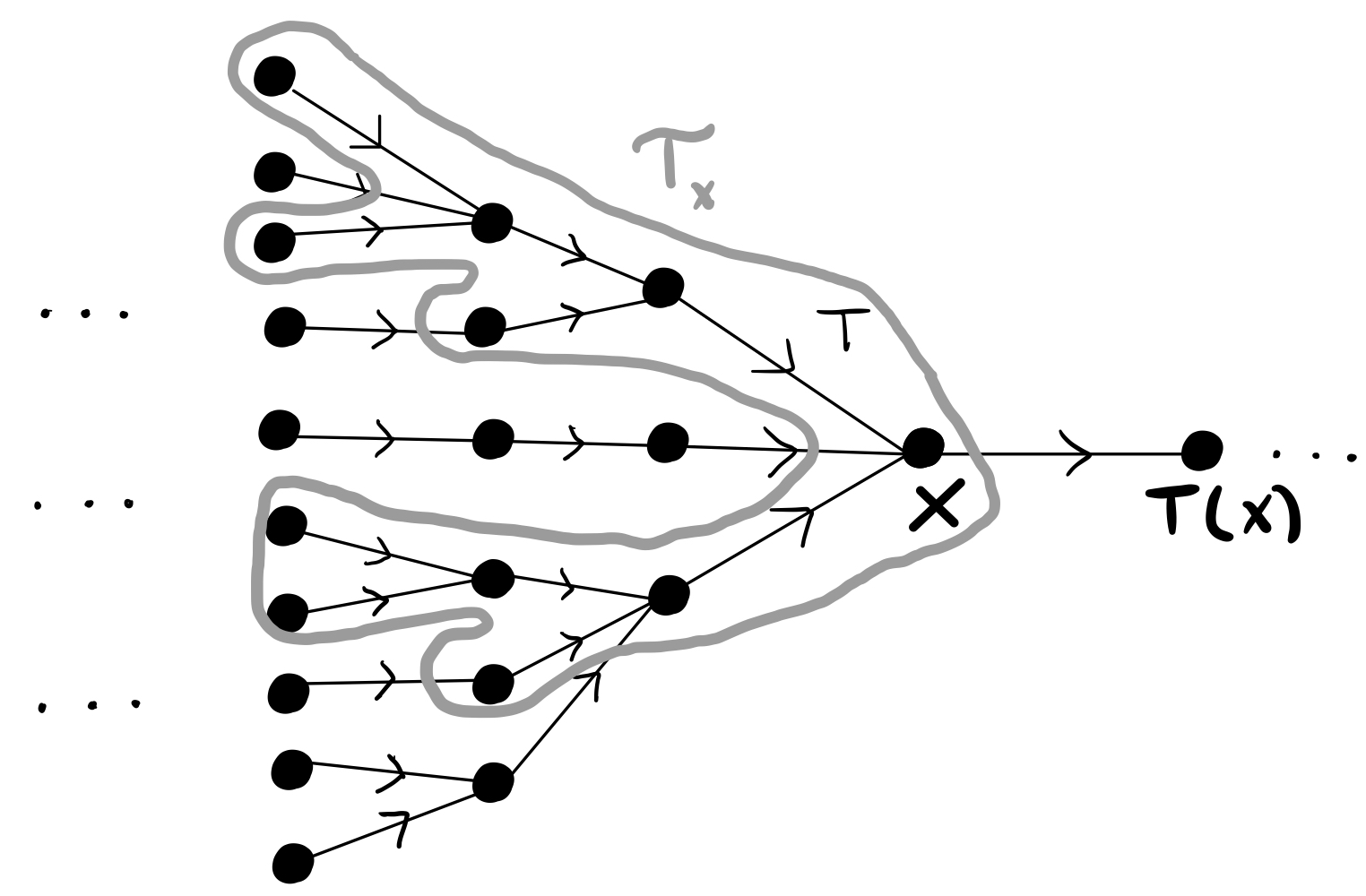}
    \caption{A partially drawn graph of $T$ with an example of $\tree_x$ circled}
    \label{fig:tree}
\end{figure}

Thus, while the classical pointwise ergodic theorem for $T$ says that to approximate $\finv$, we can start almost anywhere in the space and walk forward in time (in the direction of $T$), \cref{intro:backward_erg:trees} allows us to walk back in time (in the direction of $T^{-1}$) scanning sufficiently heavy trees of possible pasts. Note that \cref{intro:backward_erg:trees}, in particular, implies the classical (forward) pointwise ergodic theorem for one-to-one transformations $T$ when applied to $T^{-1}$. 

We also prove a \textbf{backward maximal ergodic theorem} over arbitrary subtrees (see \cref{maximal_ergodic}), but we do not use it in our proof of \cref{intro:backward_erg:trees}.

We obtain \cref{intro:forward_erg:trees} and \cref{intro:forward_erg:boundary} by applying \cref{intro:backward_erg:trees} to specific choices of $T$. In addition to \cref{intro:forward_erg:trees} and \cref{intro:forward_erg:boundary},
we also explore other applications of \cref{intro:backward_erg:trees} to the shift map on spaces of infinite words in \cref{sec:Markov_chains}. We recall the class of Markov measures on these spaces that are shift-invariant (namely, stationary Markov measures), and point out that for each such measure, \cref{intro:backward_erg:trees} allows us to calculate the expectation of an $L^1$ function by looking at trees of past trajectories of the Markov process.

\cref{intro:backward_erg:trees} implies, in particular, convergence of the $\coc$-weighted averages along any sequence $(\tree_n)$ of subtrees with $\coc(\tree_n) \to \infty$. 

We also obtain convergence in $L^p$, for all $p \ge 1$, for sequences of so-called \emph{fat} trees, see \cref{backward_erg:fat-trees}. The authors do not know whether the averages over \emph{all} trees converge in $L^p$, see \cref{q:all_trees}. An obvious example of a sequence of fat trees is that of complete trees of height $n$, i.e.\ $\bigcup_{i = 0}^n T^{-i}(x)$, and we now state this important special case (later restated as \cref{backward_erg:triangles}).

\begin{cor}[Backward ergodic along complete trees]\label{intro:backward_erg:triangles}
Let $T$ be an aperiodic countable-to-one measure-preserving transformation on a standard probability space $(X,\mu)$. Let $E_T$ denote the induced orbit equivalence relation and let $(x,y) \to \coc_x(y) : E_T \to \R^+$ be the Radon--Nikodym cocycle of $E_T$ with respect to $\mu$. For any $1 \leq p < \infty$ and $f \in L^p(X,\mu)$,
\[
\lim_{n \to \infty} \frac{1}{n+1} \sum_{i = 0}^n \sum_{y \in T^{-i}(x)} f(y) \coc_x(y) = \finv(x) \;\text{ a.e.\ and in $L^p$},
\]
where $\finv$ is the conditional expectation of $f$ with respect to the $\sigma$-algebra of $T$-invariant Borel sets.
\end{cor}

In the case where $f$ is bounded, \cref{intro:backward_erg:triangles} can be deduced directly from the classical (forward) pointwise ergodic theorem, so the new content of \cref{intro:backward_erg:triangles} is for unbounded $f$. 

Unlike \cref{intro:backward_erg:trees}, \cref{intro:backward_erg:triangles} has an operator-theoretic formulation. The operator $\backop : L^1 \to L^1$ defined by $\backop(f)(x) \defeq \sum_{y \in T^{-1}(x)} f(y) \coc_x(y)$ is nothing but the adjoint $K_T^*$ of the Koopman representation $K_T$ of $T$ (\cref{adjoint_calc}), so \cref{intro:backward_erg:triangles} simply states the convergence to $\finv$ of $\frac{1}{n+1} \sum_{i = 0}^n (K_T^*)^i(f)$, while the classical pointwise ergodic theorem states the same but for $K_T$. The mere convergence of $\frac{1}{n+1} \sum_{i = 0}^n (K_T^*)^i(f)$, not specifically to $\finv$, is already implied by \cite{DunSch} (see \cref{remark_operators}), so the contribution of \cref{intro:backward_erg:triangles} is the fact that this limit is $T$-invariant (and hence equal to $\finv$) and not just $K_T^*$-invariant.


\subsection{Context and history}\label{subsec:intro:history}

In general, a measure-preserving action of a countable (discrete) semigroup $G$ on a standard probability space $(X,\mu)$ is said to have the \dfn{pointwise ergodic property} along a sequence $(F_n)$ of finite subsets of $G$, if for every $f \in L^1(X,\mu)$, for a.e.\ $x \in X$,
\[
\lim_{n \to \infty} \frac{1}{|F_n|} \sum_{g \in F_n} f(g \cdot x) = \finv,
\]
where $\finv$ is the conditional expectation of $f$ with respect to the $\sigma$-algebra of $G$-invariant Borel sets, so in case the action is ergodic, the limit is just $\int_X f d\mu$.

\subsubsection{For pmp actions}

It is a celebrated theorem of Lindenstrauss \cite{Lindenstrauss} that the pointwise ergodic property is true for the pmp actions of all countable amenable groups along tempered \Folner sequences and this was extended by Butkevich in \cite{Butkevich} to all countable left-cancellative amenable semigroups.

Amenability, or rather the fact that the \Folner sets $F_n$ are almost invariant, is essential for the pointwise ergodic property as it ensures that the limit of averages is an invariant function. This is why, to obtain a version of the pointwise ergodic property for nonamenable (semi)groups, e.g. for the nonabelian free groups $\F_r$, one has to imitate the almost invariance of finite sets by taking weighted averages instead, so that the weight of the boundary is small. The first instance of this was proven by Grigorchuk \cite{Gri87,Gri99,Gri00}, and independently by Nevo (for $L^2$ functions \cite{Nevo}), and by Nevo and Stein \cite{NS}:

\begin{theorem}[Grigorchuk 1987; Nevo--Stein 1994]
\label{forward_erg:triangles:simple}
Let $r < \infty$ and let $\F_r \actson^\alpha (X,\mu)$ be a (not necessarily free) pmp action of the free group $\F_r$. For any $f \in L^1(X,\mu)$, for a.e.\ $x \in X$,
\[
\lim_{n \to \infty} \frac{1}{n+1} \sum_{w \in B_n} f(w \cdot x) \m_u(w) = \finv(x),
\]
where $B_n$ is the (closed) ball of radius $n$ in the standard symmetric (left) Cayley graph of $\F_r$, $\m_u$ is the uniform Markov measure\cref{ftnt:nonbacktracking-rw} on $\F_r$, and $\finv$ is the conditional expectation of $f$ with respect to the $\sigma$-algebra of $\alpha$-invariant Borel sets.
\end{theorem}

This theorem of Grigorchuk and of Nevo and Stein is a special case of our \cref{intro:forward_erg:trees}. Indeed, one simply applies \cref{intro:forward_erg:trees} to the sequence of balls $B_n$ of radius $n$ (in the standard Cayley graph of $\F_r$), observing that $\m(B_n) = n+1$. 
It is worth noting that our proof of \cref{intro:forward_erg:trees} does not use \cref{forward_erg:triangles:simple}; in fact, it provides a new proof of this result.

\cref{forward_erg:triangles:simple} was vastly generalized by Grigorchuk in \cite{Gri99} and Bufetov in \cite[Theorem 1]{Buf2} to finitely generated free semigroups (and to a large class of stationary Markov measures in Bufetov's theorem).

The aforementioned theorems of Grigorchuk, Nevo--Stein, and Bufetov are proven using the Dunford--Schwartz theorem \cite{DunSch}, an ergodic theorem for Markov operators. Our proof of \cref{intro:forward_erg:trees} follows a similar sort of reduction, but to our \cref{intro:backward_erg:trees} (instead of Dunford--Schwartz) applied to an auxiliary transformation on $X \times \partial\F_r$, which is referred to as the backward system in \cite[Section 1.5]{BQ}. That is, our proof technique is self-contained. Furthermore, the proofs in \cite{Gri87,Gri99,Gri00}, \cite{NS}, and \cite{Buf2} use Markov operators, but \textit{this technique cannot be extended to yield our results with averages over trees}, since trees do not correspond to iterates of operators.

Other instances of pointwise ergodic theorems are known for pmp actions of finitely generated groups where the averages are taken over sets that are not trees. For example, \cite{NS} and \cite{Buf} include that for pmp actions of free groups of finite rank, the averages over spheres of even radius in the Cayley graph converge a.e.\ if the function is in $L^p$ for $p > 1$ (Nevo--Stein) or even in $L \log L$ (Bufetov), however this is not true for functions in $L^1$ in general as shown by Tao \cite{Tao}. A more general treatment of pointwise ergodic theorems for groups is given in \cite{BN1} and \cite[Theorems 6.2 and 6.3]{BN2}. We refer to \cite{BK} for a survey of pointwise ergodic theorems for groups. 

\subsubsection{For nonsingular (null-preserving) actions}

There is a suitable analogue of the pointwise ergodic property for merely nonsingular\footnote{A measurable action of a countable semigroup $G$ on a probability space $(X,\mu)$ is called \dfn{nonsingular} (or \dfn{null-preserving}, or \dfn{quasi-pmp}) if for each $g \in G$, $g_* \mu \sim \mu$; equivalently, the $g$-preimage of a null set is null.} actions of (semi)groups on a standard probability space: the averages have to be weighted by the corresponding Radon--Nikodym cocycle \cite[Section 8]{KM}. Much less is known for such actions: a pointwise ergodic theorem for $\Z^d$ was first proven by Feldman \cite{Feldman:ratio_ergodic} and then generalized in two different directions by Hochman \cite{Hochman:ratio_ergodic_for_Z^d} and by Dooley and Jarrett \cite{Dooley-Jarrett:nonsingular_ergodic_Z^d_along_retangles}. For general groups of polynomial growth, Hochman obtained \cite[Theorem 1.4]{Hochman:ratio_ergodic}, a slightly weaker form of a nonsingular ergodic theorem where the a.e.\ convergence is replaced with a.e.\ convergence in density.

On the negative side, Hochman proved \cite[Theorem 1.1]{Hochman:ratio_ergodic} that the pointwise ergodic theorem for null-preserving actions holds only along sequences of subsets of the group satisfying the so-called Besicovitch covering property. He then infers \cite[Theorems 1.2 and 1.3]{Hochman:ratio_ergodic} that the null-preserving pointwise ergodic theorem fails for any sequence of subsets of $\bigoplus_{n \in \N} \Z$ and any subsequence of balls in nonabelian free groups as well as in the Heisenberg group. 

In \cite{Tserunyan:graph_ergodic}, Tserunyan obtained a pointwise ergodic theorem for locally countable null-preserving Borel graphs (equivalently, for the Schreier graphs of null-preserving actions of countable groups). Given the negative results mentioned above, this is, perhaps, the most general result in this vein.

\subsubsection{Hybrid setting}

Our main result (\cref{intro:backward_erg:trees}) is a contribution to the ergodic theory of actions of finitely generated (semi)groups in both the pmp and null-preserving settings. As mentioned above, although \cref{intro:backward_erg:trees} is about a pmp transformation $T$, its orbit equivalence relation $E_T$ is generally not pmp, and assuming, as we may, that $E_T$ is null-preserving, the ergodic averages are weighted with the corresponding Radon--Nikodym cocycle. Furthermore, \cref{intro:backward_erg:trees} implies an ergodic theorem for the natural action of the free group $\F_r$ on its boundary (\cref{intro:forward_erg:boundary}), which is merely null-preserving. On the other hand, \cref{intro:backward_erg:trees} also implies \cref{intro:forward_erg:trees}, which is about pmp actions.


\subsection{A word on the proof of the backward ergodic theorem}\label{subsec:intro:proof}

The pointwise ergodic property equates the global condition of ergodicity with the local (pointwise) statistics of the action. The key in connecting the global analysis to the local statistics is that one can often replace integrals of functions with integrals of local averages of these functions over certain shapes. For example, for a pmp group action, it easily follows from the change of variable formula that $\int f\;d\mu=\int \frac{1}{|F|}\sum_{\gamma\in F}f(\gamma\cdot x)\;d\mu(x)$ for any finite subset $F$ of the group. In our case, we have $\int f\;d\mu=\int \frac{1}{\coc_x(\triangleright_T^n \cdot x)} \sum_{y \in \triangleright_T^n \cdot x} f(y) \coc_x(y) \;d\mu(x)$ for $n \in \N$, where $\triangleright_T^n \cdot x \defeq \bigcup_{i \le n} T^{-i}x$ (see \cref{bridge:pmp}). We call such statements \emph{local-global bridges}. In operator-theoretic language, these are assertions that certain local averaging operators are Markov.

Local-global bridges reduce proving pointwise ergodic theorems to solving finitary tiling problems. These reductions are done as follows: we assume for a contradiction that the pointwise ergodic theorem fails, which gives us ``bad" tiles for each point. Then, we tile the shapes from the local-global bridge with bad tiles, thus making the integral of the function over the whole space incorrect, a contradiction. 

This scheme of proving ergodic theorems first appears implicitly in \cite{KP}, more explicitly in \cite{Tserunyan:classical_ergodic}, and even more explicitly in \cite{BZ} and in \cite{Tserunyan:graph_ergodic}.  
In the proof of \cref{intro:backward_erg:trees}, we tile sets of the form $\triangleright_T^n \cdot x$ with tiles of the form $\tree_y$, where $\tree_y$ is an arbitrary subtree of the graph of $T$ of finite height rooted at $y \in \triangleright_T^n \cdot x$ and directed towards $y$ (see \cref{triangle} for when $T$ is the shift map $\shift$ on $2^\N$). 

\begin{figure}[htp]
    \centering
    \includegraphics[width=7cm]{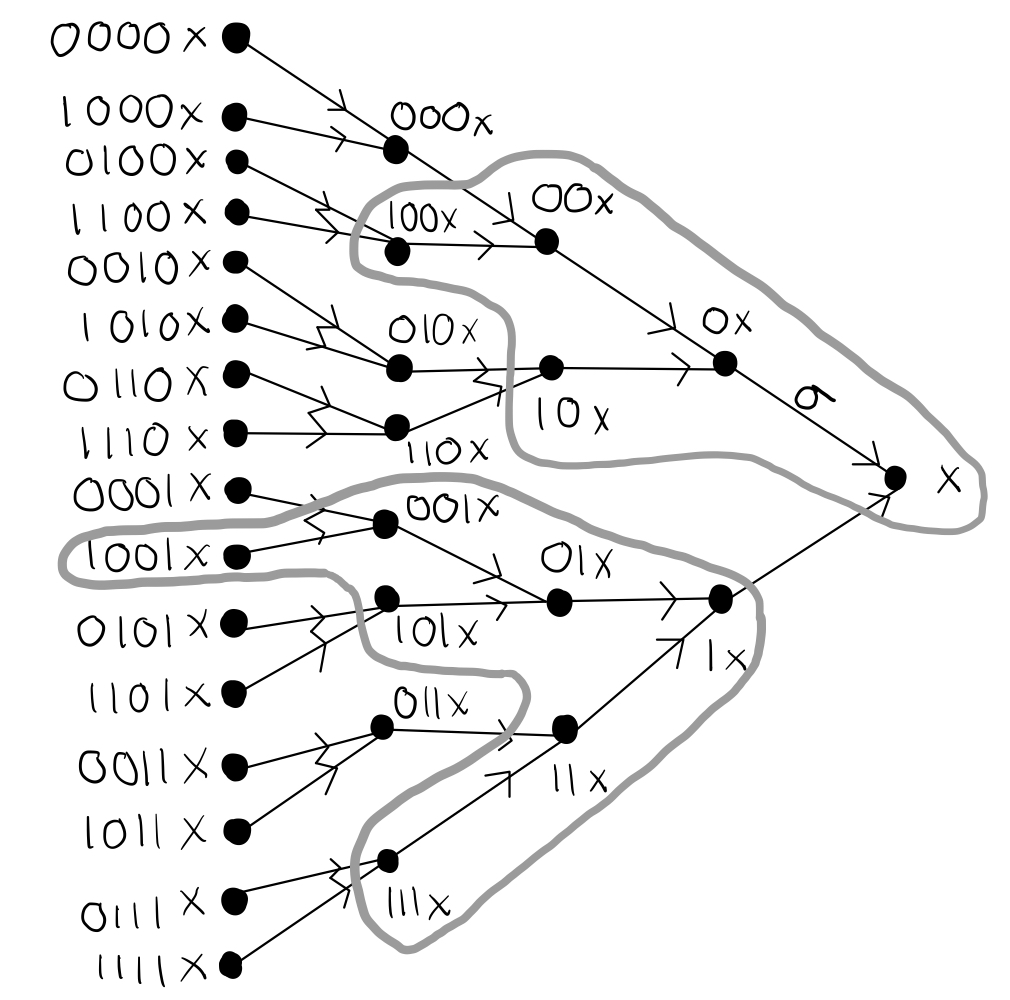}
    \caption{$\triangleright_\shift^4 \cdot x$ with examples of $\tree_x$ and $\tree_{1 x}$ circled}
    \label{triangle}
\end{figure}

For the above scheme to work, we need the limit (or rather, the $\limsup$) of local averages to be invariant (i.e. constant on each orbit).
Unlike the classical ergodic theorem, this is not clear a priori in \cref{intro:backward_erg:trees}, or especially in its special case \cref{intro:backward_erg:triangles}. Comparing the averages over the trees $\tree_x$ and $\tree_x \cup \set{T(x)}$ only gives that the $\limsup$ is nondecreasing in the direction of $T$, so one has to apply Poincar\'{e} recurrence to deduce that the $\limsup$ is constant on the orbit of $x$.



\subsubsection*{Organization}
In \cref{sec:prelim}, we give the necessary notation and definitions that are used throughout the paper, and we prove some preliminary lemmas about countable-to-one pmp Borel transformations $T$. In \cref{sec:bridge}, we state and prove the local-global bridge lemmas (\cref{bridge:null-preserving,bridge:pmp}).
In \cref{sec:theorem}, we explicitly state and prove the suitable tiling property (\cref{tiling2}) and deduce \cref{intro:backward_erg:trees} from it.
In \cref{sec:Markov_chains}, we provide examples of countable-to-one pmp ergodic Borel transformations (to which \cref{intro:backward_erg:trees} applies), and deduce the corresponding ergodic theorems, in particular \cref{intro:forward_erg:boundary} for the boundary action of the free groups.
In \cref{sec:pmp_actions}, we deduce \cref{intro:forward_erg:trees} for pmp actions of free groups.

\begin{acknowledge}
The authors thank Gabor Elek for encouraging them to investigate the pointwise ergodic property for the boundary action of the free groups. Also, thanks to Florin Boca, Clinton Conley, and Grigory Terlov for helpful discussions, and to Alexander Kechris and Aristotelis Panagiotopoulos for inviting the second author to speak in the Caltech seminar, which sparked many of the questions addressed in the paper. 
We thank Anton Bernshteyn for pointing out that \cref{intro:backward_erg:triangles} for bounded functions follows from the classical pointwise ergodic theorem (for one pmp transformation).
We thank Marius Junge for encouraging the authors to include the maximal ergodic theorem for backward trees.
Thanks to Federico Rodriguez Hertz for computing the Radon-Nikodym cocycle for the Gauss map.
We are grateful to Ronnie Chen for inspiring discussion and for many useful comments and corrections.
We owe a huge acknowledgement to Lewis Bowen for many useful suggestions and insight, which led us to the results in \cref{sec:pmp_actions}. We are also immensely thankful for the support and encouragement of Alexander Bufetov. Finally, the authors thank each other so that our acknowledgement includes women.
\end{acknowledge}

\section{Preliminaries}\label{sec:prelim}

Our set $\N$ of natural numbers includes $0$ and our $L^p$ spaces are real.

Throughout, let $(X,\mu)$ be a standard probability space. A \dfn{countable Borel equivalence relation} on $(X,\mu)$ is an equivalence relation that is a Borel subset of $X^2$ whose every equivalence class is countable. By the Luzin--Novikov uniformization theorem \cite[Theorem 18.10]{Kechris:classical_DST}, if $B\subseteq X$ is Borel, then so is its \dfn{$E$-saturation} $[B]_E \defeq \bigcup_{x\in B}[x]_E$.

\subsection{The Radon--Nikodym cocycle}

We say that a countable Borel equivalence relation $E$ on $(X,\mu)$ is \dfn{$\mu$-preserving} (resp. \dfn{null-preserving}) if for any partial Borel injection $\gamma : X \partialto X$ with $\graph(\gamma) \subseteq E$, $\mu(\dom(\gamma))=\mu(\im(\gamma))$ (resp. $\dom(\gamma)$ is $\mu$-null if and only if $\im(\gamma)$ is $\mu$-null). 

Note that $E$ is null-preserving if and only if the $E$-saturations of null sets are null. By \cite[Section 8]{KM}, a null-preserving $E$ admits an a.e.\ unique \dfn{Radon--Nikodym cocycle} $\coc : E \to \R^+$ with respect to $\mu$. Being a \dfn{cocycle} for a function $(x,y) \mapsto \coc_x(y) : E \to \R^+$ means that it satisfies the cocycle identity: 
\[
\coc_x(y) \coc_y(z) = \coc_x(z),
\]
for all $E$-equivalent $x,y,z \in X$. We say that $\coc$ is the \dfn{Radon--Nikodym cocycle} with respect to $\mu$ (or that $\mu$ is \dfn{$\coc$-invariant}) if it is Borel (as a real-valued function on the standard Borel space $E$) and for any partial Borel injection $\gamma : X \partialto X$ with $\graph(\gamma) \subseteq E$ and $f \in L^1(X,\mu)$,
\begin{equation}\label{cocycle} 
\int_{\im(\gamma)} f(x)\, d\mu(x) = \int_{\dom(\gamma)} f(\gamma(x))\coc_x(\gamma(x))\, d\mu(x). 
\end{equation}

We call a subset $\finS$ of an $E$-equivalence class $C$ \dfn{$\coc$-finite} if $\coc_x(\finS) \defeq \sum_{y \in \finS} \coc_x(y) < \infty$ for some $x \in C$ (although the value $\coc_x(\finS)$ depends on the choice of $x$, its finiteness does not, by the cocycle identity). Further, for a nonempty $\coc$-finite $\finS \subseteq C$ and a non-negative function $f : X \to [0,\infty)$, we define the \dfn{$\coc$-weighted average} of $f$ over $\finS$ by
\[
A_f^\coc[\finS] \defeq \frac{\sum_{y\in \finS} f(y) \coc_x(y)}{\coc_x(\finS)}
\]
for some $x \in C$. Again, this value does not depend on the choice of $x\in C$ by the cocycle identity. We also use the same notation for a general real-valued function $f$, provided $\sum_{y \in \finS} |f(y)| \coc_x(y) < \infty$. 

\subsection{Null-preserving orbit equivalence relations}

We say that a Borel transformation $T : X \to X$ is \dfn{$\mu$-preserving} (resp. \dfn{null-preserving}) if $T_* \mu = \mu$ (resp. $T_* \mu \sim \mu$). Let $E_T$ denote the induced \dfn{orbit equivalence relation} on $X$, that is: 
\[
x E_T y \defequiv \exists n,m, \; T^n(x)=T^m(y).
\]
Note that if $T$ is countable-to-one, $E_T$ is countable (i.e.\ each $E$-class is countable).

Even when $T$ is $\mu$-preserving, $E_T$ may not be $\mu$-preserving since $T$ may not be injective. In fact, $E_T$ may not even be null-preserving (it is possible for the $T$-image of a null set to have positive measure). However, the following result (originally proven by Kechris using Woodin's argument for the analogous statement for Baire category) shows that we may neglect this issue. Three different proofs of this are given in \cite[Proposition 2.1]{Mil1} and \cite[Proposition 1.3]{Mil2}, and we give a fourth one here, which is a measure-exhaustion argument.

\begin{lemma}[Kechris--Woodin]\label{null-preserving_on_conull}
Let $E$ be a countable Borel equivalence relation on a standard probability space $(X,\mu)$. Then $E$ is null-preserving when restricted to some conull set.
\end{lemma}

\begin{proof} 
Using the Feldman--Moore theorem \cite{FM}, fix a countable set $\Delta$ of Borel involutions $\delta : X \to X$ such that $E = \bigcup_{\delta \in \Delta} \graph(\delta)$.

\begin{claim*}
For a Borel set $Y \subseteq X$, and $\delta\in \Delta$, let $\delta_Y\defeq \delta\cap Y^2$, i.e., the restriction of $\delta$ to the set $\set{y\in Y\cap\dom(\delta):\delta(y)\in Y}$. Then $E \rest{Y}$ is null-preserving if for each $\delta \in \Delta$ and each Borel $B \subseteq \dom(\delta_Y)$, $\mu(B) > 0$ implies $\mu(\delta_Y(B)) > 0$.
\end{claim*}
\begin{pf}
Let $B \subseteq Y$ be a set whose saturation $[B]_{E \rest{Y}}$ has positive measure. Then since $[B]_{E \rest{Y}} = \bigcup_{\delta \in \Delta} \delta_Y(B\cap \dom(\delta_Y))$, we must have $\mu(\delta_Y(B\cap\dom(\delta_Y))) > 0$ for some $\delta \in \Delta$, which implies $\mu(B)\geq \mu(\delta_Y^2(B\cap\dom(\delta_Y))) > 0$, by the assumption. 
\end{pf}

To construct a conull set $Y$ satisfying the hypothesis of the claim, it is enough to fix $\delta \in \Delta$ and find a conull set $Y_\delta$ such that for each Borel $B \subseteq \dom(\delta_{Y_\delta})$, $\mu(B) > 0$ implies $\mu(\delta_{Y_\delta}(B)) > 0$ (because then $Y \defeq \bigcap_{\delta \in \Delta} Y_\delta$ is as desired).

To this end, fix $\delta \in \Delta$. We recursively construct a decreasing sequence $(X_n)$ of conull sets and a pairwise disjoint sequence $(B_n)$ of Borel subsets $B_n \subseteq X_n$ as follows. Let $X_0 \defeq X$ and suppose $X_n$ has been constructed. If there is a Borel set $B \subseteq \dom{\delta_{X_n}} $ with $\mu(B)>0$ and $\mu(\delta_{X_n}(B))=0$, let $B_n$ be one such set with 
\begin{equation}\label{eq:half-max_measure}
    \mu(B_n) > \frac{1}{2} \sup \set{\mu(B) : B \subseteq \dom{\delta_{X_n}} \text{ Borel with } \mu(\delta_{X_n}(B )) = 0}. 
\end{equation}

Then put $X_{n+1} \defeq X_n \setminus \delta_{X_n}(B_n)$. Having constructed these sequences of sets, we check that the conull set $X_\infty \defeq \bigcap_n X_n$ is as desired. Indeed, let $B \subseteq \dom(\delta_{X_\infty})$ be a Borel set with $\mu(\delta_{X_\infty}(B))=0$. Then for all $n$, \labelcref{eq:half-max_measure} implies $\mu(B_n) > \frac{1}{2} \mu(B)$, so $\mu(B) = 0$ because $\mu(B_n) \to 0$, since the $B_n$ are pairwise disjoint and $\mu(X) < \infty$.
\end{proof}

\begin{prop}\label{WLOG}
For any countable-to-one null-preserving Borel transformation $T$ on a standard probability space $(X,\mu)$, there is a conull set $X' \subseteq X$ such that $T(X') = X'$ and $E \rest{X'}$ is null-preserving.
\end{prop}
\begin{proof}
Let $X_0 \subseteq X$ be a conull set given by \cref{null-preserving_on_conull}, i.e.\ $E_T \rest{X_0}$ is null-preserving. Furthermore, because $T$ is null-preserving, the set $X_1 \defeq \bigcap_{n \in \N} T^{-n}(X_0)$ is still conull, but we now have $T(X_1) \subseteq X_1$. Lastly, again because $T$ is null-preserving, the $T$-image of a conull set is conull, so $Z \defeq X_1 \setminus T(X_1)$ is null. Hence, $[Z]_{E_T}$ is also null (because $E_T \rest{X_1}$ is null-preserving), and therefore $X' \defeq X_1 \setminus [Z]_{E_T}$ is still conull, but now we finally have $T(X') = X'$.
\end{proof}

\begin{assumption}\label{assumption}
Since all statements in the current paper are modulo null sets, without loss of generality (by \cref{WLOG}), we assume that all countable-to-one null-preserving Borel transformations $T$ on $(X,\mu)$ are surjective and the induced equivalence relations $E_T$ are null-preserving. Thus, we let $\coc$ be the Radon--Nikodym cocycle of $E_T$ with respect to $\mu$.
\end{assumption}

\subsection{Right inverses}\label{subsec:right-inverses}

Call a set $\Gamma$ of Borel partial functions $\gamma : X \partialto X$ a \dfn{complete set of Borel partial right inverses} of $T$ if the graphs of the $\gamma \in \Gamma$ are pairwise disjoint and for each $x \in X$,
\[
T^{-1}(x) = \set{\gamma(x) : x \in \dom(\gamma) \text{ and } \gamma \in \Gamma}.
\]
Because $T$ is countable-to-one, the Luzin--Novikov Uniformization theorem \cite[Theorem 18.10]{Kechris:classical_DST} ensures that $T$ admits a countable complete set $\Gamma$ of Borel partial right inverses, and we fix such a (countable) $\Gamma$ for the remainder of this subsection.

For each $n \in \N$, we think of $\Gamma^n$ as words in $\Gamma$ of length $n$. For each word $t\in \Gamma^n$, we define the partial function $t : X \partialto X$ by recursion on the length of $t$ as follows: if $t=\emptyset$, put $t(x)\defeq x$ for all $x\in X$. If $t = \gamma \conc t'$, $x \in \dom(t')$, and $t'(x) \in \dom(\gamma)$, define $t(x) \defeq \gamma(t'(x))$. Otherwise, leave $t(x)$ undefined.

\begin{obs}
If $\finS \subseteq \Gamma^n$ for some $n\in\N$, then the sets $t(X)$, $t \in \finS$, are pairwise disjoint.
\end{obs}

\noindent For $x \in X$ and $\finS \subseteq \Gamma^{< \N}$, define
\[
\finS \cdot x \defeq \set{t(x) : t \in \finS \text{ and } x \in \dom(t)}.
\]
Also, if $n\in\N$, define
\[
\triangleright_T^n \cdot x \defeq  \bigcup_{i \le n} T^{-i}(x) = \Gamma^{\le n}\cdot x.
\]
Notice that $\triangleright_T^n\cdot x$ does not depend on the choice of $\Gamma$.

For any nonempty $\finS \subseteq \Gamma^{< \N}$ and for all $f \in L^1(X,\mu)$, let $\^\finS f(x)$ be the weighted-average over the sets $\finS \cdot x$, for $x \in X$, i.e.\
\[
\^\finS f(x) \defeq A^\coc_f[\finS \cdot x].
\]

\subsection{Ergodic decomposition and conditional expectation}

This subsection is only used in \cref{sec:pmp_actions}. Let $E$ be a null-preserving countable Borel equivalence relation on a standard probability space $(X,\mu)$ and let $\coc$ be the Radon--Nikodym cocycle of $E$ with respect to $\mu$. Let $P(X)$ denote the standard Borel space of all Borel probability measures on $X$.

\begin{defn}\label{defn:ergodic_dec}
A Borel map $\epsilon : X \to P(X)$, $x \mapsto \epsilon_x$, is called an \dfn{$E$-ergodic decomposition} of $\mu$ if for each $x \in X$, the measure $\epsilon_x$ is $E$-ergodic, $\coc$-invariant, supported on $\epsilon^{-1}(\epsilon_x)$, and for all Borel sets $A \subseteq X$,
\[
\mu(A) = \int_X \epsilon_x(A) \;d\mu(x).
\]
\end{defn}

It follows from this definition that an $E$-ergodic decomposition, if it exists, is unique modulo a $\mu$-null set. As for the existence, for pmp countable Borel equivalence relations this is due to Farrel and Varadarajan \cite{Farrell,Varadarajan} (see also \cite[Theorem 3.3]{KM}), and more generally, it is a theorem of Ditzen for null-preserving equivalence relations \cite{Ditzen:thesis} (see also \cite[Theorem 5.2]{Miller:meas_with_cocycle_II}). We will only use the existence of an ergodic decomposition for pmp equivalence relations, but we will use the following connection with conditional expectation for null-preserving equivalence relations.

\begin{prop}[Conditional expectation via ergodic decomposition]\label{erg_dec--cond_exp}
Let $\epsilon : X \to P(X)$ be an $E$-ergodic decomposition of $\mu$. For any $f \in L^1(X,\mu)$ and $E$-invariant Borel set $B \subseteq X$,
\begin{equation}\label{eq:integral_over_inv-set}
\int_B f \;d\mu = \int_B \int_X f(z) \;d\epsilon_x(z) \;d\mu(x).
\end{equation}
In particular, for $\mu$-a.e.\ $x \in X$,
\[
\finv(x) = \int_X f \;d\epsilon_x,
\]
where $\finv$ is the $\mu$-conditional expectation with respect to the $\sigma$-algebra of $E$-invariant Borel sets.
\end{prop}
\begin{proof}
The ``in particular'' part follows from the definition and uniqueness of the conditional expectation and the fact that the map $x \mapsto \int_X f(z) \;d\epsilon_x(z)$ is $E$-invariant.

As for the main part, a standard approximation argument gives that for each $h \in L^1(X,\mu)$,
\begin{equation}\label{eq:erg_dec}
\int_X h \;d\mu = \int_X \int_X h(z) \;d\epsilon_x(z) \;d\mu(x).
\end{equation}

\begin{claim*}
For every $E$-invariant Borel set $B \subseteq X$, $\epsilon_x(B) = \1_B(x)$ for $\mu$-a.e.\ $x \in X$.
\end{claim*}
\begin{pf}
Indeed, letting $X_B \defeq \set{x \in X : \epsilon_x(B) = 1}$, we see that
\[
\mu(X_B \setminus B) = \int_X \epsilon_x(X_B \setminus B) \;d\mu(x) = \int_{X_B} \epsilon_x(X_B \setminus B) \;d\mu(x) + \int_{X \setminus X_B} \epsilon_x(X_B \setminus B) \;d\mu(x) = 0,
\]
because for each $x \in X_B$, $\epsilon_x(B) = 1$, so $\epsilon_x(X_B \setminus B) = 0$, and for each $x \notin X_B$, $X_B \cap \epsilon^{-1}(\epsilon_x) = \0$, so $\epsilon_x(X_B) = 0$, in particular, $\epsilon_x(X_B \setminus B) = 0$. The same argument, but with $X \setminus B$ in place of $B$, shows that $\mu(B \setminus X_B) = 0$.
\end{pf}

This and \labelcref{eq:erg_dec} applied to $h \defeq f \cdot \1_B$ imply \labelcref{eq:integral_over_inv-set}:
\begin{align*}
\int_X f \cdot \1_B \,d\mu 
&= 
\int_X \int_X f(z) \cdot \1_B(z) d\epsilon_x(z) d\mu(x) 
\\
&= 
\int_B \int_X f(z) \cdot \1_B(z) d\epsilon_x(z) d\mu(x) 
+ 
\int_{X \setminus B} \int_X f(z) \cdot \1_B(z) d\epsilon_x(z) d\mu(x)
\\
&=
\int_B \int_X f(z) d\epsilon_x(z) d\mu(x) 
+ 
0.\qedhere
\end{align*}
\end{proof}

\subsection{Limits as weight goes to infinity}\label{subsec:weight-limits}

For a set $\@T$ of objects (for us, it would be trees of various kinds), a weight-function $\m : \@T \to [0,\infty)$, a function $g : \@T \to \R$, and $L \in [-\infty, \infty]$, we write 
\begin{equation*}
g(\tree) \to L \;\text{ as }\; \m(\tree) \to \infty, 
\; 
\text{where $\tree$ ranges over $\@T$}
\end{equation*}
to mean that for every $\e > 0$, there is $M > 0$ such that for all $\tree \in \@T$ with $\m(\tree) \ge M$ we have $|g(\tree) - L| < \e$. We refer to $L$ as $\lim_{\m(\tree) \to \infty} g(\tree)$, where $\tree$ ranges over $\@T$. We also write
\begin{equation*}
\limsup_{\m(\tree) \to \infty} g(\tree) = L, 
\;
\text{where $\tree$ ranges over $\@T$}   
\end{equation*}
if $L$ is the limit of $\sup \set{g(\tree) : \tree \in \@T \text{ and } \m(\tree) \ge M}$ as $M \to \infty$; and similarly, for the $\liminf$. We omit writing ``where $\tree$ ranges over $\@T$'' when it is clear from the context.

\section{The local-global bridge}\label{sec:bridge}

Let $T : X \to X$ be a countable-to-one Borel transformation on a standard probability space $(X,\mu)$, and fix a countable complete set $\Gamma$ of Borel partial right inverses of $T$.

\subsection{Markov operators}

Before proving local-global bridge lemmas, we define and recall convenient operator-theoretic terminology. We call a linear operator $P : L^1(X,\mu) \to L^1(X,\mu)$
\begin{itemize}
    \item \dfn{non-negative}, denoted $P \ge 0$, if $Pf \ge 0$ for each non-negative $f \in L^1(X,\mu)$;
    
    \item \dfn{mean-preserving} if $\int_X Pf \; d\mu = \int_X f \;d\mu$ for each $f \in L^1(X,\mu)$;
    
    \item \dfn{Markov} if it is non-negative, mean-preserving, and $P 1 = 1$;
    
    \item an \dfn{$L^p$-contraction}, for $1 \le p \le \infty$, if $\norm{Pf}_p \le \norm{f}_p$ for each $f \in L^p(X,\mu)$.
\end{itemize}

\begin{prop}\label{mean-preserving_L1-contraction}
Every non-negative mean-preserving linear operator $P$ is an $L^1$-contraction. In fact, for $1 \le p < \infty$, if $(P |f|)^p \le P (|f|^p)$ for all $f \in L^p(X,\mu)$, then $P$ is an $L^p$-contraction.
\end{prop}
\begin{proof}
We only prove the second assertion as it subsumes the first. Fix $f \in L^p(X,\mu)$. Because $P(|f| \pm f) \ge 0$, linearity implies $|Pf| \le P |f|$. Then $\norm{Pf}_p^p =\norm{|Pf|^p}_1 \le \norm{(P|f|)^p}_1 \le \norm{P(|f|^p)}_1 = \norm{|f|^p}_1 = \norm{f}_p^p$, where the penultimate equality is by mean-preservation.
\end{proof}

\begin{prop}\label{nonnegative-fixing-1_Linf-contraction}
Every non-negative linear operator $P$ with $P1=1$ is an $L^\infty$-contraction.
\end{prop}
\begin{proof}
For each $f \in L^\infty(X,\mu)$, $P (\norm{f}_\infty \pm f) \ge 0$, so $|P f| \le P \norm{f}_\infty = \norm{f}_\infty$.
\end{proof}

\subsection{Local-global bridge for null-preserving $T$}

Throughout this subsection suppose that $T$ is $\mu$-null-preserving. In addition, we assume without loss of generality that $T$ satisfies \cref{assumption}, and we let $\coc$ be the Radon--Nikodym cocycle of $E_T$ with respect to $\mu$. 

For $f \in L^1(X,\mu)$ and $x \in X$, let
\begin{equation}\label{eq:backop}
\backop f(x) \defeq \sum_{y\in T^{-1}x}f(y)\coc_x(y).
\end{equation}

\begin{lemma}\label{avg_column}
For each $n \in \N$, $\finS \subseteq \Gamma^n$, and $f \in L^1(X,\mu)$,
\[
\int_{X} \sum_{y \in \finS \cdot x} f(y) \coc_x(y) \, d\mu(x) = \int_{\im(\finS)} f \, d\mu,
\]
where $\im(\finS) = \bigsqcup_{t \in \finS} t(X)$ \emph{(}$\bigsqcup$ denotes a disjoint union\emph{)}. In particular, $P_T$ is a non-negative mean-preserving operator, and hence an $L^1$-contraction.
\end{lemma}

\begin{proof}
That $\backop$ is mean-preserving (and hence an $L^1$-contraction by \cref{mean-preserving_L1-contraction}) is a special case of the main part with $\finS \defeq \Gamma$ because $\Gamma \cdot x = T^{-1}x$. As for the main part, it is enough to prove it for non-negative functions, since then the statement for an arbitrary  $f\in L^1(X,\mu)$ follows by the decomposition $f = f^+ - f^-$ into positive and negative parts.
To this end, we compute:
\begin{align*}
    \int_{X} \sum_{y \in \finS \cdot x} f(y)\coc_x(y) \, d\mu(x) &= \int_X \sum_{t\in \finS} \mathbb{1}_{\dom(t)}(x)f(t(x))\coc_x(t(x)) \, d\mu(x)\\
    \eqcomment{because $f \ge 0$}
    &= \sum_{t\in \finS}\int_X \mathbb{1}_{\dom(t)}(x)f(t(x))\coc_x(t(x)) \, d\mu(x)\\
    \eqcomment{by \cref{cocycle}}
    &= \sum_{t\in \finS}\int_{\text{im}(t)}f(x) \, d\mu(x) 
    \\
    &= \int_{\text{im}(\finS)} f(x) \, d\mu(x).
    \qedhere
\end{align*}
\end{proof}

Let $K_T$ denote the Koopman operator on $L^1(X,\mu)$ induced by $T$, i.e. $K_T f \defeq f\circ T$ for $f \in L^1(X,\mu)$. We explicitly calculate its adjoint $K_T^*$.

\begin{prop}\label{adjoint_calc}
The operator $\backop$ is the adjoint $K_T^*$ of the Koopman operator $K_T$; more precisely, for all $f,g\in L^1(X,\mu)$ such that $K_T f \cdot g \in L^1(X,\mu)$,
\[
\int_X K_T f\cdot g\; d\mu = \int_X f \cdot \backop g\;d\mu.
\]
\end{prop}
\begin{proof}
We compute:
\begin{align*}
\int_X K_T(f)(x)\cdot g(x)\; d\mu(x) 
&= 
\int_X (f\circ T)(x)g(x) d\mu(x)
\\
\eqcomment{\cref{avg_column}} 
&= 
\int_X \sum_{y\in T^{-1}(x)}(f\circ T)(y)g(y)\coc_x(y) \; d\mu(x)
\\
&= 
\int_X f(x) \sum_{y\in T^{-1}(x)}g(y)\coc_x(y) \; d\mu(x)
\\
&= 
\int_X f(x) \backop(g)(x)\; d\mu(x).
\qedhere
\end{align*}
\end{proof}

The following observation follows easily by induction on $n$ and the cocycle identity.

\begin{obs}\label{iterated_backop}
$\backop^n = \backop[T^n]$ for each $n \in \N$.
\end{obs}

\begin{cor}\label{sum_over_columns_is_finite}
For any $f \in L^1(X,\mu)$, for a.e.\ $x \in X$, and for all $n \in \N$,
\[
\sum_{y \in \triangleright_T^n \cdot x} |f(y)| \coc_x(y) < \infty.
\]
\end{cor}
\begin{proof}
This is just because $\sum_{y \in \triangleright_T^n \cdot x} |f(y)| \coc_x(y) = \sum_{i = 0}^n \backop[T^i] |f|(x)$ and each $\backop[T^i]$ maps $L^1(X,\mu)$ to $L^1(X,\mu)$ (\cref{avg_column}), so $\big\|\sum_{i = 0}^n \backop[T^i] |f|\big\|_1 < \infty$.
\end{proof}

\begin{lemma}[Local-global bridge for null-preserving $T$]\label{bridge:null-preserving}
Let $T$ be a countable-to-one null-preserving Borel transformation on $(X,\mu)$. For any $N \in \N$ and $f\in L^1(X,\mu)$,
\[
\int_X f\; d\mu = \int_X \frac{1}{N+1} \sum_{y \in \triangleright_T^N\cdot x} f(y)\coc_x(y) \, d\mu(x).
\]
\end{lemma}

\begin{proof}
Observing that $\sum_{y \in \triangleright_T^N\cdot x} f(y)\coc_x(y) = \sum_{n = 0}^N \backop[T^n] f(x)$, the statement follows from the fact that $\backop[T^n]$ is mean-preserving (\cref{avg_column}).
\end{proof}

Recall that for any nonempty $\finS \subseteq \Gamma^{< \N}$ and for all $f \in L^1(X,\mu)$, $\^\finS f(x)$ is the weighted-average over the sets $\finS \cdot x$, for $x \in X$, i.e.\
\[
\^\finS f(x) = A^\coc_f[\finS \cdot x].
\]
\cref{sum_over_columns_is_finite} shows that $\^\finS f$ is well-defined. Moreover:

\begin{cor}\label{S_avg_Lp-bounded}
Let $n \in \N$ and $\finS \subseteq \Gamma^{\le n}$ be such that the function $x \mapsto \coc_x(\finS \cdot x)$ is bounded below by some $w > 0$. For any $1 \le p \le \infty$, $\^\finS$ is a bounded operator on $L^p(X,\mu)$ with operator norm $\|\^\finS\|_p \le \left(\frac{n+1}{w}\right)^{1/p}$.
\end{cor}

\begin{proof}
For $p = \infty$, the statement is obvious, so suppose $p < \infty$. For any $f \in L^p(X,\mu)$,
\begin{align*}
    \|\^\finS f\|_p^p
    =& 
    \int_X |A^\coc_f[\finS \cdot x]|^p \;d \mu(x)
    \\
    \le&
    \int_X A^\coc_{|f|}[\finS \cdot x]^p \;d \mu(x)
    \\
    \eqcomment{Jensen's inequality}
    \le&
    \int_X A^\coc_{|f|^p}[\finS \cdot x] \;d \mu(x)
    \\
    =&
    \int_X \frac{1}{\coc_x(\finS \cdot x)} \sum_{y \in \finS \cdot x} |f(y)|^p \coc_x(y) \;d \mu(x)
    \\
    \le&
    \frac{n+1}{w} \int_X \frac{1}{n+1} \sum_{y \in \triangleright_T^n \cdot x} |f(y)|^p \coc_x(y) \;d \mu(x)
    \\
    \eqcomment{\cref{bridge:null-preserving}}
    =&
    \frac{n+1}{w} \norm{f}_p^p.
    \qedhere
\end{align*}
\end{proof}

\subsection{Local-global bridge for pmp $T$}

Throughout this subsection, we assume in addition that $T$ preserves the measure $\mu$.

\begin{lemma}\label{each_column_weight=1}
For a.e.\ $x \in X$, $\coc_x(T^{-n}x) = 1$ for each $n \in \N$; in other words, $\backop[T^n] 1 = 1$. In particular, $\coc_x(\triangleright_T^N\cdot x) = N+1$ for each $N \in \N$.
\end{lemma}

\begin{proof}
The second statement is immediate from the first. For the first statement, we may switch the quantifiers, i.e.\ prove that for each $n \in \N$, the formula holds a.e.\ By \cref{iterated_backop}, it is enough to prove the statement for $n=1$. 

To this end, we show that for each $\e > 0$, the set $Z_\e \defeq \set{x\in X: \coc_x(T^{-1}(x)) > 1 + \e}$ is null. This implies that $\set{x\in X: \coc_x(T^{-1}(x)) > 1}$ is null and an analogous argument shows that $\set{x\in X: \coc_x(T^{-1}(x)) < 1}$ is null as well. Since $T$ is $\mu$-preserving,
\begin{align*}
    \mu(Z_\e) 
    = 
    \mu({T^{-1}(Z_\e)})
    &= 
    \mu\left(\bigsqcup_{\gamma \in \Gamma} \gamma(Z_\e \cap \dom(\gamma))\right)
    \\
    &= 
    \sum_{\gamma \in \Gamma} \mu(\gamma(Z_\e \cap \dom(\gamma)))
    \\
    \eqcomment{by \cref{cocycle}}
    &=
    \sum_{\gamma \in \Gamma} \int_{Z_\e \cap \dom(\gamma)} \coc_x(\gamma(x))\; d\mu(x)
    \\
    &= 
    \int_{Z_\e} \sum_{\gamma \in \Gamma} \mathbb{1}_{\dom(\gamma)}(x) \coc_x(\gamma(x)) \; d\mu(x)
    \\
    &= 
    \int_{Z_\e} \sum_{y \in T^{-1}(x)} \coc_x(y) \; d\mu(x)
    \\
    &= 
    \int_{Z_\e} \coc_x(T^{-1}(x))\;d\mu(x) 
    \ge
    \mu(Z_\e) (1 + \e).
\end{align*}
This implies that $Z_\e$ is null, as desired.
\end{proof}

\cref{sum_over_columns_is_finite,each_column_weight=1,bridge:null-preserving} together immediately yield:

\begin{cor}[Local-global bridge for pmp $T$]\label{bridge:pmp}
Let $T$ be a countable-to-one pmp Borel transformation on $(X,\mu)$. For any $N \in \N$ and $f\in L^1(X,\mu)$, $A_{|f|}^\coc [\triangleright_T^N\cdot x] < \infty$ a.e.\ and
\[
\int_X f(x)\, d\mu(x) = \int_X A_f^\coc [\triangleright_T^N \cdot x]\, d\mu(x).
\]
\end{cor}

\begin{remark}
\cref{bridge:pmp} fails when we replace $\triangleright_T^N \cdot x$ with arbitrary subsets of the back-orbit of $x$, even trees (as in \cref{subsec:trees_tiling}). We may observe this by looking at indicator functions of the images of right inverses $\gamma$ of $T$. For example, if $T$ is the shift map on $(2^\N,\set{\frac{1}{2},\frac{1}{2}}^\N)$, and $f\defeq\mathbb{1}_{\set{x(0)=0}}$, then $\int f\;d\mu=\frac{1}{2}$, but $\int A_f^\coc\left[\set{x,0 \conc x}\right]\;d\mu= \frac{2}{3}$.
\end{remark}

\begin{cor}\label{backop_Markov_Lp-contraction}
The operator $\backop$ is Markov and an $L^p$-contraction for all $1 \le p \le \infty$.
\end{cor}
\begin{proof}
$\backop$ is Markov by \cref{avg_column,each_column_weight=1}, so it is an $L^\infty$-contraction (\cref{nonnegative-fixing-1_Linf-contraction}). Furthermore, \cref{each_column_weight=1} makes Jensen's inequality applicable for $1 \le p < \infty$, yielding $(\backop |f|)^p \le \backop |f|^p$, so \cref{mean-preserving_L1-contraction} applies.
\end{proof}

It is convenient to define the weighted averages over the sets $\triangleright_T^N$ as operators: for each $N \in \N$, $f \in L^1(X,\mu)$, and $x \in X$, define
\begin{equation}\label{eq:triangle_avg}
    \triangle_{T,N} f(x) \defeq A_f^\coc [\triangleright_T^N \cdot x].
\end{equation}

\cref{each_column_weight=1,iterated_backop} immediately imply:

\begin{cor}\label{avg-backop=triangle}
$\triangle_{T,N} = \frac{1}{N+1} \sum_{n = 0}^N \backop[T^n] = \frac{1}{N+1} \sum_{n = 0}^N \backop^n$ for each $N \in \N$.
\end{cor}

This and \cref{backop_Markov_Lp-contraction} imply:

\begin{cor}\label{triangle_Markov_Lp-contraction} 
For each $N \in \N$, the operator $\triangle_{T,N}$ is Markov and an $L^p$-contraction for all $1 \le p \le \infty$.
\end{cor}

\section{The tiling property and the backward ergodic theorem}\label{sec:theorem}

Throughout, let $(X,\mu)$ be a standard probability space and let $T:X\to X$ be a countable-to-one $\mu$-preserving Borel transformation, so by \cref{assumption}, $T$ is surjective and $E_T$ is null-preserving. Let $\coc : E_T \to \R^+$ be the Radon--Nikodym cocycle with respect to $\mu$. Finally, let $\Gamma$ be a complete set of Borel partial right inverses of $T$.

\subsection{The tiling property}\label{subsec:trees_tiling}

We now prove the needed tiling property and deduce our backward pointwise ergodic theorem (\cref{backward_erg:trees}) from it.

For a set $\Symb$ (which will typically be a countable complete set of Borel right inverses of a transformation $T$), let $\@T_\Symb \subseteq \@P(\Symb^{< \N})$ be the set of nonempty \dfn{set-theoretic} (but \dfn{right-rooted}) trees on $\Symb$ of finite height, where $\Symb^{< \N}$ is the set of all finite sequences of elements of $\Symb$. More precisely, for each $\tree \subseteq \Symb^{< \N}$, 
\begin{align*}
    \tree \in \@T_\Symb 
    \defequivlong
    &
    \tree \text{ is of finite height, i.e. }
    \tree \subseteq \@P(\Symb^{\le n}) 
    \text{ for some $n$, }
    \\
    &
    \tree \text{ contains the empty word }\emptyset,
    \\
    &
    \text{and for each $t_1,t_2\in \Symb^{<\N}$, if $t_1 t_2 \in \tree$, then $t_2$ is also in $\tree$}.
\end{align*}
For each $\tree \in \@T_\Symb$, denote by $h(\tree)$ the \dfn{height} of the tree $\tree$, i.e.\ the least $n \in \N$ such that $\tree \subseteq \Symb^{\le n}$.

\begin{lemma}[Tiling property]\label{tiling2}
Let $T$, $\coc$, and $\Gamma$ be as above.
Then for any measurable function $x \mapsto \tree_x : X \to \@T_\Gamma$ and $\e > 0$ there is $N \in \N$ and a set $X' \subseteq X$ of measure $\ge 1-\e$ such that for all $x \in X'$, the complete tree $\triangleright_T^N\cdot x$ can be covered, up to $\e$ fraction of its $\coc_x$-weight, by disjoint tiles of the form $\tree_y \cdot y$.

More precisely, for every $x \in X'$ there is a subset $S_x$ of $\triangleright_T^N \cdot x$ with $\rho_x(S_x) \ge (1-\e) \rho_x(\triangleright_T^N \cdot x)$ that is partitioned into sets of the form $\tree_y \cdot y$ for $y \in X$.
\end{lemma}

\begin{proof}
Let $L$ be large enough so that the set 
\[
B \defeq \set{x \in X: h(\tree_x) \geq L}
\]
has measure less than $\frac{\e^2}{2}$. Fix $N$ large enough so that $\frac{L}{N}<\frac{\e}{2}$. By \cref{each_column_weight=1}, we may assume that for each $x \in X$ and $n \in \N$, $\sum_{y \in T^{-n}(x)} \coc_x(y) = 1$, so 
\[
\coc_x(\triangleright_T^{N-L} \cdot x) = N-L > (1 - \frac{\e}{2}) N =  (1 - \frac{\e}{2}) \coc_x(\triangleright_T^N \cdot x).
\]
Thus, there is no harm in leaving $\triangleright_T^N\cdot x \setminus \triangleright_T^{N-L}(x)$ untiled.

We claim that for all but less than $\e$-measured set of $x \in X$, less than $\frac{\e}{2}$ $\coc_x$-fraction of $y \in \triangleright_T^N\cdot x$ are in $B$, i.e.\ the set
\[
C \defeq \set{x \in X: A_{\mathbb{1}_B}^\coc[\triangleright_T^N\cdot x] \geq \frac{\e}{2}},
\]
has measure less than $\e$. Indeed:
\begin{align*}
    \frac{\e^2}{2} 
    > 
    \mu(B)
    &= 
    \int_X \mathbb{1}_B(x) \, d\mu(x)
    \\
    \eqcomment{by \cref{bridge:pmp}}
    &= 
    \int_X A_{\mathbb{1}_B}^\coc[\triangleright_T^N\cdot x]\, d\mu(x)
    \\
    &\ge 
    \int_C A_{\mathbb{1}_B}^\coc[\triangleright_T^N\cdot x]\, d\mu(x)
    \\
    &\ge
    \frac{\e}{2}\mu(C).
\end{align*}

So we just need to fix $x \in X \setminus C$ and tile the set $\triangleright_T^N\cdot x$ up to an $\e$ $\coc_x$-fraction. 
We do this by the following straightforward algorithm (see \cref{triangle} in \cref{subsec:intro:proof}): if there is $n \le N$ and $y \in T^{-n}(x)$ that is not covered by a tile yet and $\tree_y \cdot y \subseteq \triangleright_T^{N}(x)$, take the least such $n$ and for each such $y \in T^{-n}(x)$, place the tiles $\tree_y \cdot y$; repeat this until there is no such $n$. 
Once this process terminates, the only points that are not covered by a tile must belong to either $B$ or $\triangleright_T^N\cdot x \setminus \triangleright_T^{N-L} \cdot x$, so they comprise at most $\e$ $\coc_x$-fraction of $\triangleright_T^N \cdot x$. 
\end{proof}

\subsection{Poincar\'e recurrence}

Here, we recall some ergodic-theoretic terminology and basic facts, which are used in \cref{subsec:ptwise_erg_along_trees}. A set $W \subseteq X$ is called \dfn{$T$-wandering} if the sets $T^{-n}(W)$, $n \in \N$, are pairwise disjoint. Because the measures of the sets $T^{-n}(W)$ are all equal and $\mu$ is a probability measure, we have:

\begin{obs}\label{conservativity}
$T$ is conservative, i.e., every $T$-wandering measurable set is null.
\end{obs}

For a set $U\subseteq X$, let
\[
[U]_T^+ \defeq \bigcup_{n \in \N^+} T^n(U),
    \hspace{1em}
    [U]_{T^{-1}}^+ \defeq \bigcup_{n \in \N^+} T^{-n}(U),
    \hspace{1em}
    [U]_{T^{-1}} \defeq \bigcup_{n \in \N} T^{-n}(U).
\]
Abusing notation, we write $[x]^+_T$ when $U= \set{x}$. Note that for any set $U$, the set $V \defeq X \setminus [U]_{T^{-1}}$ is closed under $T$, i.e.\ $T(V) \subseteq V$.

Call a set $U \subseteq X$ \dfn{$T$-recurrent} if for every $x \in U$, $[x]_T^+ \cap U\ne \0$; equivalently, $U\setminus [U]_{T^{-1}}^+ = \0$. 
Consequently, we say that a set $U \subseteq X$ is \dfn{$\mu$-nowhere $T$-recurrent} if it does not admit a $T$-recurrent subset of positive measure.

\begin{lemma}[Poincar\'e recurrence]\label{T-recurrence}
Every Borel set $U\subseteq X$ is $T$-recurrent a.e.\ In fact, there is a subset $U' \subseteq U$ that is conull in $U$ such that for every $x \in [U']_{E_T}$, $[x]_T^+ \cap U' \ne \0$.
\end{lemma}
\begin{proof}
The set
$
U'' \defeq \set{x \in U: [x]_T^+ \cap U = \0}
$
is $T$-wandering and hence null. Then $[U'']_{T^{-1}}$ is also null, and it is easy to check that $U' \defeq U \setminus [U'']_{T^{-1}}$ is as desired.
\end{proof}

In light of \cref{T-recurrence}, we may assume that all positively measured sets that come up are $T$-recurrent. 

\begin{lemma}\label{T-invariance}
Every Borel set $U\subseteq X$ with the property that $T(U)\subseteq U$ is such that $[U]_{E_T}=U$ off of a $T$-invariant null set.
\end{lemma}

\begin{proof}
Put $V\defeq [U]_{E_T}\setminus U$. Then since $T(U)\subseteq U$, $V$ is $\mu$-nowhere $T$-recurrent because for all $x\in V$, there are only finitely many $n$ with $T^n(x)\in V$. Hence, by \cref{T-recurrence}, $V$ is null, and since $E_T$ is null-preserving, so is $[V]_{E_T}$. Therefore, $U$ is $E_T$-invariant off of the invariant null set $[V]_{E_T}$.
\end{proof}

We say that the \dfn{periodic part} of $T$ is the subset $\set{x\in X:\exists n<m\in\N: \; T^n(x)=T^m(x)}.$

\begin{lemma}
\label{Periodicity}
$T$ is bijective on its periodic part off of a null set. 
\end{lemma}

\begin{proof}
Let $V$ be the periodic part of $T$, and let $U\defeq \set{x\in X:\exists n\in\N\setminus\set{0}:\; T^n(x)=x}$. Then $[U]_{E_T}=V$. Notice that $V\setminus U$ is nowhere $T$-recurrent, hence null, and that $T\rest{U}$ is bijective. 
\end{proof}

\subsection{Backward ergodic theorem along trees}\label{subsec:ptwise_erg_along_trees}

We think of $\graph(T)$ as a directed graph on $X$, where $X$ is the set of vertices and $\graph(T)$ is the set of directed edges.
For $x\in X$, let $\@T_x$ denote the collection of subtrees of $\graph(T)$ of finite height rooted at $x$ and directed towards $x$ (see \cref{fig:tree} in \cref{subsec:intro:results}). More precisely, $\tree_x \in \@T_x$ exactly when the following three conditions hold:
\begin{enumerate}[(i)]
    \item $\tree_x \subseteq \bigcup_{i=0}^nT^{-i}(x)$ for some $n \in \N$;

    \item $x \in \tree_x$;

    \item if  $y \in \tree_x$ and $y \neq x$ then $T(y) \in \tree_x$.
\end{enumerate}

Notice that if $\Gamma$ is a complete set of Borel partial right inverses of $T$, then $\tree_x \in \@T_x$ exactly when $\tree_x = \tree \cdot x$ for some $\tree \in \@T_\Gamma$. With this in mind, we use graph-theoretic trees ($\tree_x \subseteq \graph(T)$) and set-theoretic trees ($\tree \subseteq \Symb^{< \N}$) interchangeably in the rest of the paper.

\begin{lemma}\label{invariance_of_limsup}
Let $T$ and $\coc$ be as above, and additionally assume that $T$ is aperiodic. Then for any $f \in L^1(X,\mu)$, for a.e.\ $x\in X$, for all $\tree_x \in \@T_x$, we have that  $A_{|f|}^\coc[\tree_x]<\infty$ and
the functions
\[
\fsup(x) \defeq \limsup_{\coc_x(\tree_x) \to \infty} A_f^\coc[\tree_x] 
\text{ and } 
\finf(x) \defeq \liminf_{\coc_x(\tree_x) \to \infty} A_f^\coc[\tree_x],\; \text{where $\tree_x$ ranges over $\@T_x$},
\]
are $T$-invariant a.e.\ (i.e.\ off of a $T$-invariant null set).
\end{lemma}
\begin{proof}
That $A_{|f|}^\coc[\tree_x]<\infty$ for a.e.\ $x \in X$ and each $\tree_x \in \@T_x$, is by \cref{sum_over_columns_is_finite} because $f \in L^1(X,\mu)$.
Thus, we may assume without loss of generality that this holds for all $x \in X$, and in particular, $A_{f}^\coc[\tree_x]$ is well-defined for all $x \in X$.

As for invariance, it is enough to show that $\fsup$ is $T$-invariant as $\finf = - \fsup[(-f)]$. For that, it is enough to show that for each $a \in \Q$, the set 
\[
X_{\ge a} \defeq \set{x \in X : \fsup(x) \ge a} 
\]
is $T$-invariant, modulo a null set. 
Fix $a\in\Q$. 
By \cref{T-invariance}, we just need to show $T(X_{\ge a})\subseteq X_{\ge a}$. 
For this, it suffices to show that $\fsup(x)\leq \fsup(T(x))$ for all $x \in X_{\ge a}$.
Fix $x \in X_{\ge a}$ and put $y \defeq T(x)$. 

Intuitively, since $y$ is only one point, we can add it to heavy trees $\tree \in \@T_x$ without having much impact on the weighted average, hence $\fsup(x)\le \fsup(y)$. 
To see this more formally, recall that $\fsup(x) \ge a > -\infty$ and fix an arbitrary real $S < \fsup(x)$, weight $w > 0$, and error $\e > 0$.
It is enough to find $\tree_{y} \in \@T_{y}$ such that $\coc_x(\tree_{y}) \ge \coc_x(y) \cdot w$ (equivalently, $\coc_{y}(\tree_{y}) > w$) and $A_f^\coc[\tree_{y}] \geq S - \e$.

To this end, take $\tree_x \in \@T_x$ of large enough $\coc_x$-weight so that $\coc_x(\tree_x) \ge \coc_x(y) \cdot w$ and
\[
|S| + |f(y)| \le \frac{\coc_x(\tree_x)}{\coc_x(y)} 
\cdot 
\e
\]
and $A_f^\coc[\tree_x] \ge S$. 
Putting $\tree_{y} \defeq \tree_x \sqcup \set{y}$ (hence, $\tree_{y} \in \@T_{y}$), we have:
\begin{align*}
A_f^\coc[\tree_{y}] 
&=
\left(1 - \frac{\coc_x(y)}{\coc_x(\tree_{y})}\right)
\cdot
A_f^\coc[\tree_x] 
+ 
\frac{\coc_x(y)}{\coc_x(\tree_{y})} 
\cdot
f(y) 
\\
&\ge
S 
- 
\frac{\coc_x(y)}{\coc_x(\tree_{y})} \cdot |S|
- 
\frac{\coc_x(y)}{\coc_x(\tree_{y})} 
\cdot
|f(y)|
\\
&=
S 
- 
\frac{\coc_x(y)}{\coc_x(\tree_{y})} \cdot (|S| + |f(y)|)
\\
&\ge
S 
- 
\frac{\coc_x(y)}{\coc_x(\tree_{y})} 
\cdot
\frac{\coc_x(\tree_x)}{\coc_x(y)} 
\cdot
\e
\\
&\ge
S - \e.
\qedhere
\end{align*}
\end{proof}


\begin{theorem}[Backward pointwise ergodic along trees]\label{backward_erg:trees}
Let $T$ be an aperiodic countable-to-one pmp Borel transformation on a standard probability space $(X,\mu)$, so by \cref{assumption}, $T$ is surjective and $E_T$ is null-preserving. Let $(x,y)\mapsto \coc_x(y) : E_T \to \R^+$ be the Radon--Nikodym cocycle of $E_T$ with respect to $\mu$. For every $f \in L^1(X,\mu)$ and for a.e.\ $x \in X$, we have $A_{|f|}^\coc[\tree_x] < \infty$ for all $\tree_x \in \@T_x$, and
\[
A_f^\coc[\tree_x]  \to \finv(x) \;\text{ as }\; \coc_x(\tree_x)\to \infty,
\]
where $\tree_x$ ranges over $\@T_x$, and $\finv$ is the conditional expectation of $f$ with respect to the $\sigma$-algebra of $T$-invariant Borel sets. 
\end{theorem}


\begin{proof}
By \cref{invariance_of_limsup}, for a.e.\ $x\in X$ and each $\tree_x \in \@T_x$, $A_{|f|}^\coc[\tree_x]<\infty$, and $\fsup$ and $\finf$ are $T$-invariant.
By replacing $f$ with $f - \finv $, we may assume without loss of generality that $\finv = 0$. We will show that $\fsup \leq 0$ a.e., and an analogous argument shows $\finf  \geq 0$ a.e.

Assume towards a contradiction that $\fsup > 0$ on a positively measured (necessarily $T$-invariant) set, restricting to which we might as well assume that $\fsup(x) > 0$ for all $x\in X$. Put $g\defeq \min\set{\frac{\fsup}{2},1}$, so $0<g\leq 1$, and $g\in L^1(X,\mu)$. Put $c\defeq \int g\; d\mu >0$. Fix a complete set $\Gamma$ of Borel partial right inverses of $T$, and let $\@T'_\Gamma \subseteq \@T_\Gamma$ denote the collection of finite trees in $\@T_\Gamma$. Notice that for each $x\in X$, the quantity $\limsup_{\coc_x(\tree \cdot x) \to \infty} A_f^\coc[\tree \cdot x]$ does not change if we restrict the range of $\tree$ to only $\@T_\Gamma'$ since $\sum_{y\in \tree_x}f(y)\coc_x(y)$ converges absolutely for each $\tree_x\in\@T_x$. 

Fix an enumeration $\set{\tree_n}$ of $\@T'_\Gamma$, and define $\ell : X \to \@T'_\Gamma$ by $x \mapsto \tau_n$ where $n \in \N$ is least such that $A_f^\coc[\tree_n \cdot x] > g(x)$ (equivalently, $A_{f-g}^\coc[\tree_n \cdot x] > 0$).

Fix $\delta > 0$ small enough so that for any measurable $Y \subseteq X$, $\mu(Y)<\delta$ implies $\int_Y (f-g)\, d\mu > -\frac{c}{3}$, and let $M\in\N$ be large enough so that the set $Y \defeq f^{-1}(-M,\infty)$ has measure at least $1 - \delta$.

The tiling property (\cref{tiling2}) applied to the function $\ell$ with $\e \defeq \frac{1}{2(M+1)} \frac{c}{3}$ gives $N \in \N$ such that $\mu(Z) \ge 1 - \e$, where $Z$ is the set of all $x\in X$ such that at least $1-\e$ $\coc_x$-fraction of $\triangleright_T^n\cdot x$ is partitioned into sets of the form $\ell(y)\cdot y$.

\begin{claim*}
$A_{\mathbb{1}_Y(f-g)}^\coc[\triangleright_T^n\cdot x] \ge -(M+1) \e$ for each $x \in Z$.
\end{claim*}

\begin{pf}
By the definition of $Z$, on a subset $B \subseteq \triangleright_T^n\cdot x$ that occupies at least $1-\e$ $\coc_x$-fraction of $\triangleright_T^n\cdot x$, the $\coc$-average of $f - g$ is positive, and hence that of $\mathbb{1}_Y (f-g) \rest{B}$ is non-negative. On the remaining set $\triangleright_T^n\cdot x \setminus B$, the function $\mathbb{1}_Y (f-g)$ is at least $-(M+1)$, by the definition of $Y$, and hence so is its $\coc$-weighted average. Thus, the $\coc$-weighted average of $\mathbb{1}_Y (f-g)$ on the entire $\triangleright_T^n\cdot x$ is at least $- (M+1) \e$.
\end{pf}
 
\noindent Now we compute using this claim and \cref{bridge:pmp}:
\begin{align*}
    \int_Y(f-g) \, d\mu  
    &= 
    \int_X A_{\mathbb{1}_Y(f-g)}^\coc[\triangleright_T^n\cdot x]\,d\mu(x)
    \\ 
    &= 
    \int_Z A_{\mathbb{1}_Y(f-g)}^\coc[\triangleright_T^n\cdot x]\,d\mu(x) 
    + 
    \int_{X \setminus Z} A_{\mathbb{1}_Y(f-g)}^\coc[\triangleright_T^n\cdot x] \,d\mu(x) 
    \\
    &\ge 
    -(M+1) \e - (M+1) \e = -2 (M+1) \e = - \frac{c}{3}.
\end{align*}
This gives a contradiction:
\begin{align*}
    0 = \int_X \finv\, d\mu = \int_X f\, d\mu 
    &= 
    c +\int_X (f-g) \, d\mu
    \\
    &= 
    c + \int_Y(f-g) \, d\mu + \int_{X \setminus Y} (f-g) \, d\mu 
    \\
    &> 
    c - \frac{c}{3} - \frac{c}{3} > 0. \qedhere
\end{align*}
\end{proof}

\begin{remark}
It is worth explicitly pointing out particular sequences $(\tree_n)$ of trees in $\@T_\Gamma$ such that $\coc_x(\tree_n\cdot x)\to \infty$ regardless of the base point $x\in X$. Such is the sequence $\triangleright_T^n$; indeed, by \cref{each_column_weight=1}, $\coc_x(\triangleright_T^n\cdot x) = n+1$ for a.e.\ $x \in X$. More generally, this is true for sequences $(\tree_n)$ of trees that contain shifted complete trees whose heights tend to infinity. By this we mean that there is a fixed word $t \in \Symb^{<\N}$ such that $\tree_n$ contains the shifted complete tree $\triangleright_T^{h_n}t$, where $h_n\to \infty$.
\end{remark}

\begin{remark}
By \cref{Periodicity}, $T$ is bijective on its periodic part $Y$ (mod null). By \cref{assumption}, we can assume $E_T\rest{Y}$ is null-preserving and $T\rest{Y}$ is bijective. Consequently, $E_T\rest{Y}$ is measure-preserving, so the backward averages are unweighted (as in the  standard pointwise ergodic theorem for $\Z$). However, the only trees in each orbit of $T\rest{Y}$ are paths of bounded length, so their weights do not tend to infinity. We can still reformulate the statement of our theorem with set theoretic trees instead (i.e. over $\tree \cdot x$ where $\tree \in \@T_\Gamma$ for some complete set $\Gamma$ of Borel partial right inverses of $T$) and the statement would still be true without the aperiodicity assumption.
\end{remark}

If $T$ and $f$ are as in \cref{backward_erg:trees}, then for a.e.\ $x\in X$, the set $\set{\frac{1}{n}\sum_{i<n}f(T^i(x)):n\in\N}$ of forward averages is bounded (since this is a convergent sequence). Looking backward, our \cref{backward_erg:trees} also says that the averages $A_f^\coc[\tree_x]$ over $\tree_x \in \@T_x$ converge as $\coc_x(\tree_x) \to \infty$, but there are infinitely-many trees $\tree_x \in \@T_x$ of bounded weight, so the mere convergence does not imply that the set $\set{A_f^\coc[\tree_x] : \tree_x \in \@T_x}$ is bounded. Nevertheless, we show it is indeed bounded; in fact, the maximal ergodic theorem holds along backward trees (see \cite{KP} for the classical forward version). We use the boundedness in the proof of \cref{intro:forward_erg:trees}, but of course, the maximal ergodic theorem is interesting in its own right.

\begin{theorem}[Backward maximal ergodic theorem along trees]\label{maximal_ergodic}
Let $T$ be as in \cref{backward_erg:trees}. Let $f \in L^1(X,\mu)$, and define $f^*(x) \defeq \sup_{\tree_x \in \@T_x} A_f^\coc[\tree_x]$ for each $x \in X$. Then for any $\lambda \in \R$,
\[
\int_{f^*>\lambda}f\;d\mu\ge \lambda\mu\set{f^*>\lambda}.
\]
In particular, $f^* < \infty$ a.e.
\end{theorem}
\begin{proof}
Fix $\lambda \in \R$ and let $Y \defeq \set{x\in X:f^*>\lambda}$. We will show $\int_Y f\;d\mu\ge \lambda\mu(Y)$.
First note that this is equivalent to showing $\int_Y(f-\lambda)\;d\mu\ge -\e$ for arbitrary $\e>0$.

For each $x \in Y$, let $\tree_x \in \@T_x$ be a minimal witness to $x$ being in $Y$, i.e., $A_f^\coc[\tree_x]>\lambda$, and no proper subtree of $\tree_x$ has this property. Hence, for any $y\in \tree_x$, $y$ is also in $Y$ (since $\tree_x\cap \bigcup_{i\in \N}T^{-i}(y)$ has average greater than $\lambda$ by the minimality of $\tree_x$). Then $A^\coc_{\1_Y(f-\lambda)}[\tree_x]=A^\coc_{(f-\lambda)}[\tree_x]>0$ for each $x\in Y$. For $x\notin Y$, set $\tree_x\defeq\set{x}$. 

The rest of the proof is morally the same as that of \cref{backward_erg:trees}, so we will not provide all of the details. We may assume without loss of generality (by the same argument as in \cref{backward_erg:trees}) that $f$ is bounded from below. We apply the tiling property (\cref{tiling2}) to get $N$ large enough so that for each point $x$ in a set $Z$ of large measure (which will depend on the lower bound of $f$), we can tile most of $\triangleright_T^N \cdot x$ with tiles of the form $\tree_y$, which, together with the lower bound for $f$ guarantees that $A^\coc_{\1_Y(f-\lambda)}[\triangleright_T^N \cdot x]\ge -\frac{\e}{2}$. Hence, by the local-global bridge (\cref{bridge:pmp}),
\begin{align*}
\int_Y(f-\lambda)\;d\mu &= \int A^\coc_{\1_Y(f-\lambda)}(\triangleright_T^N \cdot x)\; d\mu \\
&= \int_Z A^\coc_{\1_Y(f-\lambda)}(\triangleright_T^N \cdot x)\; d\mu + \int_{X\setminus Z} A^\coc_{\1_Y(f-\lambda)}(\triangleright_T^N \cdot x)\; d\mu \\
&\ge -\frac{\e}{2}+\int_{X\setminus Z} A^\coc_{\1_Y(f-\lambda)}(\triangleright_T^N \cdot x)\; d\mu.    
\end{align*}
By taking $Z$ to have arbitrarily large measure, we get $\int_Y(f-\lambda)\; d\mu\ge-\e$.
\end{proof}

\subsection{Convergence in $L^p$ along special sequences of trees}

Besides pointwise convergence, we also get convergence in $L^p$ along the sequence of complete trees $\triangleright_T^n$. This is the content of \cref{intro:backward_erg:triangles}, which we restate and prove here. We also remark afterwards that the theorem holds for other special sequences of trees as well.

\smallskip

Recall (\cref{triangle_Markov_Lp-contraction}) that for each $n$, the operator $\triangle_{T,n}$ on $L^1(X,\mu)$ defined by $\triangle_{T,n}f(x) \defeq A_f^\coc[\triangleright_T^n \cdot x]$ is a Markov operator, which is an $L^p$-contraction for all $1 \le p \le \infty$.

\begin{cor}[Backward ergodic along complete trees]\label{backward_erg:triangles}
Let $T$ and $\coc$ be as in \cref{backward_erg:trees}. For any $1 \leq p < \infty$ and $f \in L^p(X,\mu)$,
\[
\lim_{n \to \infty} \triangle_{T,n} f = \finv \;\text{ a.e.\ and in $L^p$},
\]
where $\finv$ is the conditional expectation of $f$ with respect to the $\sigma$-algebra of $T$-invariant Borel sets.
\end{cor}

\begin{proof}
The pointwise convergence follows immediately from \cref{backward_erg:trees} by considering, for $x \in X$, the sequence $\triangleright_T^n \cdot x$ of complete trees, recalling that by \cref{each_column_weight=1}, $\coc_x(\triangleright_T^n \cdot x) = n+1 \to \infty$.

If $f \in L^\infty(X,\mu)$, then $|\triangle_{T,n} f| \le \norm{f}_\infty$ for every $n \in \N$, so by the dominated convergence theorem, $\triangle_{T,n} f$ converges in $L^p$ to $\finv$ as $n \to \infty$.

For a general $f\in L^p(X,\mu)$, let $(f_k)$ be a sequence of bounded functions converging to $f$ in $L^p$. Fix $\e>0$, and let $k$ be large enough so that $\|f-f_k\|_p<\frac{\e}{3}$. This implies, for all $n \in \N$, that $\|\triangle_{T,n} f - \triangle_{T,n} f_k\|_p < \frac{\e}{3}$ and $\|\finv - \finv_k\|_p < \frac{\e}{3}$ because both $\triangle_{T,n}$ and conditional expectation are $L^p$-contractions (\cref{triangle_Markov_Lp-contraction} and \cite[Theorem 4.1.11]{Dur}). Thus,
\begin{align*}
    \norm{\triangle_{T,n} f - \finv}_p 
    \leq& 
    \norm{\triangle_{T,n} f - \triangle_{T,n} f_k}_p 
    + 
    \norm{\triangle_{T,n} f_k - \finv[f_k]}_p
    + 
    \norm{\finv[f_k] - \finv}_p
    \\
    <& 
    \frac{\e}{3}
    + 
    \norm{\triangle_{T,n} f_k - \finv[f_k]}_p 
    + 
    \frac{\e}{3} 
    < \e,
\end{align*}
for large enough $n$ because we already know that $\lim_{n \to \infty} \norm{\triangle_{T,n} f_k - \finv[f_k]}_p = 0$.
\end{proof}

Recalling \cref{adjoint_calc,avg-backop=triangle}, we now restate \cref{backward_erg:triangles} in terms of the adjoint $K_T^*$ of the Koopman representation $K_T$ of $T$.

\begin{cor}\label{ergodic-for-adjoint}
Let $T$ be as in \cref{backward_erg:trees}. For any $1\le p<\infty$ and $f \in L^p(X,\mu)$,
\[
\lim_{n \to \infty} \frac{1}{n+1} \sum_{i=0}^n (K_T^*)^i (f) = \overline{f} \;\text{ a.e.\ and in $L^p$},
\]
where $\finv$ is the conditional expectation of $f$ with respect to the $\sigma$-algebra of $T$-invariant Borel sets.
\end{cor}

\begin{remark}\label{remark_operators}
The mere convergence of the sequence of averages in \cref{ergodic-for-adjoint} is implied by the Dunford--Schwartz ergodic theorem \cite{DunSch}:
\end{remark}

\begin{theorem}[Dunford--Schwartz 1956]\label{Dunford-Schwartz}
If $Q : L^1(X,\mu) \to L^1(X,\mu)$ is a non-negative $L^1$-$L^\infty$-contraction\footnote{This means both an $L^1$-contraction and an $L^\infty$-contraction.} on a probability space $(X,\mu)$, then for every $f \in L^1(X,\mu)$, there exists a $Q$-invariant $\hat{f} \in L^1(X,\mu)$ such that $\frac{1}{n+1} \sum_{i=0}^n Q^i f \to \hat{f}$ as $n \to \infty$ both a.e.\ and in $L^1(X,\mu)$.
\end{theorem}

\noindent Indeed, by \cref{backop_Markov_Lp-contraction,adjoint_calc} $Q \defeq K_T^*$ is an $L^1$-$L^\infty$-contraction, so the sequence $\frac{1}{n}\sum_{i<n}(K_T^*)^i(f)$ converges a.e.\ and in $L^1(X,\mu)$ to a $K_T^*$-invariant function $\hat{f}\in L^1(X,\mu)$. However, the $K_T^*$-invariance of $\hat{f}$ does not directly imply that $\hat{f}$ is $T$-invariant (and hence one cannot conclude that $\hat{f}$ is the conditional expectation $\finv$ with respect to the $T$-invariant $\sigma$-algebra of Borel sets). The main new content of \cref{ergodic-for-adjoint} is that $\hat{f}$ is indeed $T$-invariant. To prove this, we rephrased it in terms of averages over complete backward trees (\cref{backward_erg:triangles}) and proved the stronger statement of \cref{backward_erg:trees} that the averages over \textit{arbitrary} backward trees converge.

Lastly, we discuss a more general version of \cref{backward_erg:triangles}, replacing the sequence of complete backward trees with that of ``fat'' backward trees. Let $T$ and $\coc$ be as in \cref{backward_erg:trees} and let $\Gamma$ be a complete set of Borel partial right inverses of $T$. For $c > 0$, call a tree $\tree \subseteq \@T_\Gamma$ \dfn{$c$-fat} for the cocycle $\coc$ if $\frac{\coc_x(\tree \cdot x)}{h(\tree) + 1} \ge c$ for a.e.\ $x \in X$. In particular, $\Gamma^{\le n}$ is $1$-fat and \cref{backward_erg:triangles} is about the averaging operators $\^{\Gamma^{\le n}}$, defined in \cref{subsec:right-inverses}.

\begin{cor}[Backward ergodic theorem along fat trees]\label{backward_erg:fat-trees}
Let $T$, $\coc$, and $\Gamma$ be as above. Let $c > 0$ and let $(\tree_n)$ be any sequence of trees in $\@T_\Gamma$ that are $c$-fat for $\coc$ and such that $h(\tree_n) \to \infty$ as $n \to \infty$. Then for any $1 \leq p < \infty$ and $f \in L^p(X,\mu)$,
\[
\lim_{n \to \infty} \^\tree_n f = \finv \;\text{ a.e.\ and in $L^p$},
\]
where $\finv$ is the conditional expectation of $f$ with respect to the $\sigma$-algebra of $T$-invariant Borel sets.
\end{cor}
\begin{proof}
Replacing the operators $\triangle_{T,n}$ with $\^\tree_n$, the argument is word-for-word the same as for \cref{backward_erg:triangles}, except that the operators $\^\tree_n$ may not be $L^p$-contractions, but by \cref{S_avg_Lp-bounded}, their norms are uniformly bounded by $c$ (independent of $n$), which is all the argument needs.
\end{proof}

\begin{question}\label{q:all_trees}
For a complete set $\Gamma$ of Borel partial right inverses of $T$, for which sequences $(\tree_n)$ of trees in $\@T_\Gamma$ do we have convergence in $L^p$ of the $\coc$-weighted averages over $\tree_n \cdot x$?
\end{question}


\section{Applications to shift maps}\label{sec:Markov_chains}

An example of a countable-to-one Borel transformation is the shift map $\shift : \Symb^\N \to \Symb^\N$, for some countable set $\Symb$, where $\shift(x) \defeq \big(x(1 + n)\big)_{n \in \N}$. See \cref{fig:2-shift} for the depiction of $\shift$ for $\Symb \defeq 2 \defeq \set{0,1}$.

\begin{figure}[htp]
    \centering
    \includegraphics[width=7cm]{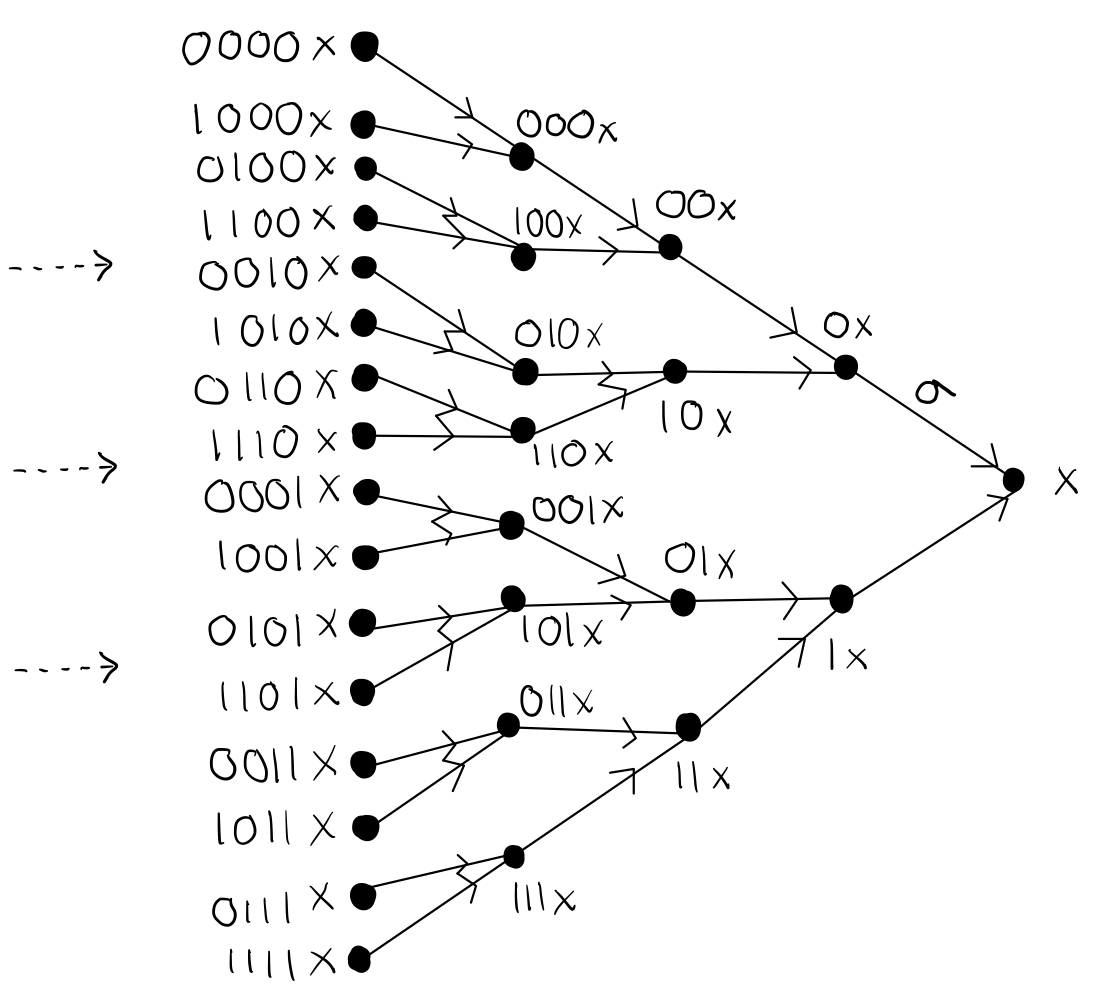}
    \caption{Shift on $2^\N$}
    \label{fig:2-shift}
\end{figure}

With the exception of the first example, the measures on $\Symb^\N$ discussed in this section will be Markov measures. Indeed, Markov measures on the state space $\Symb$ provide a large class of Borel probability measures on $\Symb^\N$, including shift-invariant ones (which are exactly those with stationary initial distribution). Thus, \cref{backward_erg:trees} says that averaging a function while walking backward in the directions according to a tree $\tree \in \@T_\Symb$ approximates the conditional expectation of the function with respect to the $\sigma$-algebra of shift-invariant Borel sets (\cref{fig:2-shift} depicts all backward walks of length $3$ for $\Symb \defeq \set{0,1}$). 


\subsection{Example:\ the Gauss map}
Let $X\defeq (0,1]$ and $\mu$ be the Gauss measure, i.e., $d\mu \defeq \frac{1}{(\log2)(1+x)}\;d\lambda$, where $\lambda$ is Lebesgue measure. Let $T : X \to X$ be the Gauss map $x\mapsto \frac{1}{x}\text{ mod }1$. For $x \in \N^\N$, we denote by $[x]$ the real in $(0,1]$ whose continued fraction expansion is $x$. The map $x \mapsto [x] : \N^\N \to (0,1]$ is an equivariant isomorphism of the shift $\shift$ on $\N^\N$ and the Gauss map $T$ on $(0,1]$.
In particular, $T$ is countable-to-one. Moreover, $T$ is $\mu$-preserving and ergodic (see \cite{Kea}), so \cref{backward_erg:trees,backward_erg:triangles,maximal_ergodic} apply. For concrete applications of this theorem, it is useful to have an explicit formula for the Radon--Nikodym cocycle of $E_T$ with respect to the Gauss measure $\mu$. For each $j \in \N^+$ and $x \in \N^\N$, let\footnote{This formula was obtained by Federico Rodriguez-Hertz by change of variable trickery.}
\[
\coc_{[x]}([j,x]) \defeq \frac{1+[x]}{([x]+j)([x]+j+1)} = (1+[x]) [j,x] [j+1, x].
\]
This induces a cocycle on $E_T$ via the cocycle identity. One can simply check that this fits the definition of the Radon--Nikodym cocycle of $E_T$ with respect to $\mu$. In particular, the adjoint $K_T^*$ of the Koopman representation is given by the formula:
\[
K_T^* f (x) = \sum_{j \in \N^+} f([j,x]) (1+[x]) [j,x] [j+1, x]
\]
for $f \in L^1(X,\mu)$ and $x \in X$ (see \cref{adjoint_calc}). Furthermore, the averages of iterates of $K_T^*$ converge to the expectation a.e.\ and in $L^p$ for all $p \ge 1$ (see \cref{ergodic-for-adjoint}). 


\subsection{Preliminaries on Markov measures}
Here we give some basics of Markov measures, referring the reader to \cite{Dur} for a more comprehensive exposition of this subject.

Let $\Symb$ be a countable discrete state space. For $w \in \Symb^{<\N}$, let $|w|$ denote the length of $w$, and for a finite or infinite word $w' \in \Symb^{<\N} \cup \Symb^\N$, let $w \conc w'$ denote the concatenation of $w$ and $w'$ (i.e.\ the word $w$ followed by $w'$), and for $\finS \subseteq \Symb^{<\N}$, let $\finS \conc w \defeq \set{v \conc w : v \in \finS}$.

An $\Symb \times \Symb$ (row) stochastic matrix\footnote{A square matrix with nonnegative entries whose rows add up to $1$.} $P$ is called \dfn{irreducible} if for each $i,j \in \Symb$ there is $n \ge 1$ such that the $(i,j)$ entry of $P^n$ is positive. Call a probability distribution $\pi$ on $\Symb$ \dfn{stationary} for the matrix $P$ if treating $\pi$ as a row-vector, we have $\pi P = \pi$.

The \dfn{Markov measure} on $\Symb^{<\N}$ with \dfn{transition matrix} $P$ and \dfn{initial distribution} $\pi$ is the measure $\m$ on $\Symb^{<\N}$ defined by 
\[
\m(w) \defeq \pi\big(w(0)\big) \cdot P\big(w(0), w(1)\big) \cdots P\big(w(\ell-2), w(\ell-1)\big),
\]
for each nonempty word $w \in \Symb^{<\N}$, and $\m(\0) \defeq 1$. Note that $\m$ is a probability distribution on $\Symb^n$ for each $n \ge 0$. We say that $\m$ is \dfn{irreducible} (resp.\ \dfn{stationary}) if $P$ is irreducible (resp.\ $\pi$ is a stationary distribution for $P$).

\begin{assumption}\label{assumption:Markov}
We assume throughout that the initial distribution $\pi$ of every Markov measure is \textbf{positive} (i.e.\ all of its entries are positive).
\end{assumption}

Moving to infinite words, we equip $\Symb^\N$ with the standard Borel structure induced by the product topology, where $\Symb$ is discrete. In particular, the cylindrical sets
\[
[w] \defeq \set{x \in \Symb^\N : w \text{ is an initial subword of } x},
\]
$w \in \Symb^{< \N}$, are clopen and form a basis for the topology. Any Markov measure $\m$ on $\Symb^{<\N}$ induces a probability measure $\P_\m$ on $\Symb^\N$ uniquely defined by $\P_\m[w] \defeq \m(w)$ for each $w \in \Symb^{<\N}$. We also refer to $\P_\m$ as a \dfn{Markov measure} on $\Symb^\N$.

The following proposition records the basic connections between the properties of $\m$ and $\P_\m$ that we use in our arguments; the proofs of these connections are standard, see \cite[5.5]{Dur}.

\begin{prop}\label{Markov-facts}
Let $\m$ be a Markov chain on $\Symb$ and let $\shift$ denote the shift map on $\Symb^\N$.

\begin{enumerate}[label=\normalfont{(\alph*)}, leftmargin=*, itemsep=4pt]
    \item \label{Markov:null-preserving} If the initial distribution of $\m$ is positive (\cref{assumption:Markov}), then $\shift$ is $\P_\m$-null-preserving if and only if the transition matrix of $\m$ does not have a zero column.
    
    \item \label{Markov:pmp} $\shift$ is $\P_\m$-preserving if and only if $\m$ is stationary.
    
    \item \label{Markov:ergodic} $\shift$ is $\P_\m$-ergodic if $\m$ is irreducible and its transition matrix admits a stationary distribution.
\end{enumerate}
\end{prop}


\subsection{Backward pointwise ergodic theorem for Markov measures}

If a Markov measure $\m$ has positive initial distribution (\cref{assumption:Markov}) and its transition matrix has no zero columns, then the shift map $\shift$ is $\P_\m$-null-preserving (\cref{Markov-facts}\labelcref{Markov:null-preserving}). 
In fact, $E_\shift$ is null-preserving on the conull set
\[
Y \defeq \set{x \in \Symb^\N : P(x(n), x(n+1)) > 0 \text{ for all } n \in \N}.
\]
On this conull set $Y$, we explicitly calculate the Radon--Nikodym cocycle of $E_\shift$ with respect to $\P_\m$.

\begin{prop}\label{correctcocycle}
Let $\m$ be a Markov chain on $\Symb$ whose initial distribution is positive (\cref{assumption:Markov}) and whose transition matrix has no zero columns. 
Then the Radon--Nikodym cocycle of $E_\shift$ with respect to $\P_\m$ is given by 
\[
\coc_{\shift^n(x)}(x) = \frac{\m\big(x(0) \conc x(1) \conc \dots \conc x(n)\big)}{\m\big(x(n)\big)}
\]
for all $n \in \N$ and a.e.\ $x \in \Symb^\N$.
\end{prop}

\begin{proof}
To see that $\coc$ is the Radon--Nikodym cocycle with respect to $\P_\m$, it suffices to check that for all $w \in \Symb^{<\N}$ and $i \in \Symb$ with $\m(i \conc w) > 0$, we have 
\[
\P_\m([i \conc w]) = \int_{[w]} \coc_x(i \conc x) \;d\P_\m(x).
\]

Let $P$ and $\pi$ be the transition matrix and the initial distribution of $\m$. 
If $w \ne \0$, then $\coc_x(i \conc x)=\frac{\pi(i)}{\pi(w(0))}P(i,w(0))$ for a.e.\ $x \in [w]$. 
Hence,
\[
    \int_{[w]}\coc_x(i \conc x)\;d\P_\m(x) 
    = 
    \frac{\pi(i)}{\pi(w(0))} \cdot P(i,w(0))\cdot \P_\m[w]
    = 
   \P_\m([i \conc w]).
\]

Lastly, if $w = \0$, then 
\begin{align*}
    \int_{[w]}\coc_x(i \conc x)\;d\P_\m(x) 
    &=
    \int_{\Symb^\N} \coc_x(i \conc x)\;d\P_\m(x)
    \\
    &=
    \sum_{j\in \Symb}\int_{[j]}\coc_x(i \conc x)\;d\P_\m(x)
    \\
    &= \sum_{j\in \Symb} \frac{\pi(i)}{\pi(j)}P(i,j)\P_\m([j])
    \\
    &= 
    \pi(i) \sum_{j\in \Symb} P(i,j) 
    \\
    &= 
    \pi(i) 
    = 
    \P_\m([i]) 
    = 
    \P_\m([i \conc w]).
    \qedhere
\end{align*}
\end{proof}

We now state \cref{backward_erg:trees,backward_erg:triangles,maximal_ergodic} for the shift map on $\Symb^\N$ with a Markov measure, using
\cref{Markov-facts}\labelcref{Markov:pmp} and \cref{correctcocycle}. Recall that $\@T_\Symb$ denotes the set of right-rooted set-theoretic trees on $\Symb$ (see \cref{subsec:trees_tiling}).


\begin{cor}[Pointwise ergodic property for Markov measures]\label{ergodic_Markov}
Let $\m$ be a stationary Markov measure on $\Symb^{<\N}$. For every $1 \le p < \infty$ and $f \in L^p(\Symb^\N,\P_\m)$, we have the following.

\begin{enumerate}[(a), leftmargin=2em,itemsep=8pt]
    \item \emph{Pointwise convergence along arbitrary trees:} 
    
    \smallskip
    
    \noindent For a.e.\ $x \in \Symb^\N$,
    \[
    \frac{1}{\m(\tree \conc x(0))}\sum_{w \in \tree}f(w\conc x)\m(w\conc x(0)) \to \finv \;\text{ as }\; \m(\tree \conc x(0))\to \infty,
    \]
where $\tree$ ranges over $\@T_\Symb$, and $\finv$ is the conditional expectation of $f$ with respect to the $\sigma$-algebra of shift-invariant Borel sets.

\item \emph{$L^p$ convergence along complete trees:} 

\smallskip

\noindent The functions $x \mapsto \frac{1}{(n+1) \m(x(0))} \sum_{w \in \Symb^{\le n}} f(w \conc x) \m(w \conc x(0))$ converge to $\finv$ both a.e.\ and in $L^p$.

\item \emph{Maximal ergodic theorem along arbitrary trees:} 

\smallskip

\noindent Letting $f^*(x) \defeq \sup_{\tree \in \@T_\Symb} \frac{1}{\m(\tree \conc x(0))}\sum_{w \in \tree}f(w\conc x)\m(w\conc x(0))$, we have
\[
\int_{f^*>\lambda}f\;d\mu\ge \lambda\mu\set{f^*>\lambda},
\]
for any $\lambda \in \R$. In particular, $f^* < \infty$ a.e.
\end{enumerate}
\end{cor}

\subsection{Bernoulli shifts}

For a finite state space $\Symb$, the simplest example of a Markov measure $\m$ on $\Symb^{<\N}$ to which \cref{ergodic_Markov} applies and for which the shift map $\shift$ is $\P_\m$-ergodic is the one whose initial distribution $\pi$ is uniform and the transition matrix is constant, i.e.\ all entries are equal to $\frac{1}{|\Symb|}$. Indeed, in this case $\P_\m$ is just the product measure $\pi^\N$ and the Radon--Nikodym cocycle of $E_\shift$ with respect to $\P_\m$ is given by $\coc_{\shift(x)}(x)\defeq \frac{1}{|\Symb|}$ for all $x \in \Symb^\N$.

We can also view the sequences $x\in \Symb^\N$ as $|\Symb|$-ary representations of $x\in [0,1)$. Thus, $\shift$ is
the same as the so-called \dfn{baker's map} $T : [0,1) \to [0,1)$ given by $x \mapsto |\Symb| \cdot x \text{ mod }1$, with Lebesgue measure on $[0,1)$.

\subsection{Boundary actions of free groups}\label{subsec:boundary_actions}

For $1 \le r \leq \infty$, let $\F_r$ be the free group on $r$ generators $\set{b_i}_{i<r}$ and let $\Symb_r \defeq \set{a_i}_{i<2r}$, where $a_{2i} \defeq b_i$ and $a_{2i+1} \defeq b_i^{-1}$ for each $i < r$. We recall that the boundary $\partial\F_r$ can be viewed as the set of all infinite reduced\footnote{A finite or infinite word $w$ on the set $\Symb_r$ is called \dfn{reduced} if a generator and its inverse do not appear side-by-side in $w$.} words in $\Symb_r$. This makes $\partial \F_r$ a closed subset of $\Symb_r^\N$. 

The group $\F_r$ has a natural (boundary) action $\F_r \actson^\beta \partial \F_r$ by concatenation and cancellation: for $w \in \F_r$ and $x \in \partial \F_r$, $w \cdot x \defeq (w \conc x)^*$, where the latter denotes the reduction of the word $w \conc x$. This action is free on the (cocountable) set of \dfn{aperiodic} words\footnote{A word on a set $\Symb$ of symbols is called \dfn{periodic} if it is of the form $w \conc v v v \dots$ for some $w,v \in \Symb^{<\N}$.}.

The main relevant fact about this action is that its orbit equivalence relation is the same as that of the shift $\shift : \partial \F_r \to \partial \F_r$; in fact, for $x \in \partial \F_r$, $\shift(x) = x(0)^{-1} \cdot x$, and conversely, for any $a \in \Symb_r$ and $x \in \partial \F_r$ with $x(0) \ne a^{-1}$, $a \cdot x = a \conc x \in \shift^{-1}(x)$. Thus, for $x \in \partial \F_r$ and $n \in \N$, we have $\triangleright_\shift^n \cdot x = B_n^{x(0)} \cdot x$, where $B_n^{x(0)}$ is the set of all reduced words of length at most $n$ that do not end with $x(0)^{-1}$. Therefore, applying \cref{ergodic_Markov} to an appropriate class of Markov measures on $\Symb_r^{<\N}$ yields a pointwise ergodic theorem (\cref{intro:forward_erg:boundary}, restated below as \cref{forward_erg:boundary}) for the boundary action $\F_r \actson^\beta \partial \F_r$.

To translate the conclusion of \cref{ergodic_Markov} into a statement about the boundary action $\F_r \actson^\beta \partial \F_r$, we need the support of the Markov measure $\P_\m$ on $\Symb_r^\N$ to be contained in $\partial \F_r$. This is the same as requiring that the support of $\m$ is contained in $\F_r$, which is equivalent to the transition matrix $P$ of $\m$ satisfying
\[
    P(a,a^{-1}) = 0 \text{ for all } a \in \Symb_r.
\]
Lastly, we denote by $\@T_{\Symb_r}^\bullet$ the set of all finite height subtrees of the (left) Cayley graph of $\F_r$ containing the identity.

\begin{cor}[Pointwise ergodic for boundary actions of free groups]\label{forward_erg:boundary}
Let $1 \le r \le \infty$ and let $\Symb_r$ be the standard symmetric set of generators of $\F_r$. 
Let $w \cdot x$ denote the boundary action of $w \in \F_r$ on $x \in \partial \F_r$. 
Let $\m$ be a stationary Markov measure on $\Symb_r^{<\N}$ whose support is contained in $\F_r$. 
For every $1 \leq p < \infty$ and $f \in L^p(\partial\F_r,\P_\m)$, we have the following.

\begin{enumerate}[(a), leftmargin=2em,itemsep=8pt]
    \item \emph{Pointwise convergence along arbitrary trees:}
    
    \smallskip
    
    \noindent For a.e.\ $x \in \partial \F_r$,
    \[
    \frac{1}{\m(\tree \conc x(0))} \sum_{w\in \tree}f(w\cdot x)\m(w\conc x(0)) \to \finv(x) \;\text{ as }\; \m(\tree) \to \infty,
    \]
    where $\tree$ ranges over all trees in $\@T_{\Symb_r}^\bullet$ not containing $x(0)^{-1}$, and $\finv$ is the conditional expectation of $f$ with respect to the $\sigma$-algebra of $\beta$-invariant Borel sets.

\item \emph{$L^p$ convergence along complete trees:} 

\smallskip

\noindent The functions $x \mapsto \frac{1}{(n+1)\m(x(0))} \sum_{w\in B_n^{x(0)}} f(w\cdot x) \m(w\conc x(0))$ converge to $\finv$ both a.e.\ and in $L^p$.

\item \emph{Maximal ergodic theorem along arbitrary trees:} 

\smallskip

\noindent Letting
\[
f^*(x) \defeq \sup \set{\frac{1}{\m(\tree \conc x(0))} \sum_{w\in \tree}f(w \cdot x)\m(w \conc x(0)) : \tree \in \@T_{\Symb_r}^\bullet \text{ and } x(0)^{-1} \notin \tree},
\]
we have
\[
\int_{f^*>\lambda}f\;d\mu\ge \lambda\mu\set{f^*>\lambda},
\]
for any $\lambda \in \R$. In particular, $f^* < \infty$ a.e.
\end{enumerate}
\end{cor}

We now provide explicit examples of Markov measures $\m$ on $\Symb_r^{<\N}$ to which \cref{forward_erg:boundary} applies and for which the boundary action $\F_r \actson^\beta \partial \F_r$ is $\P_\m$-ergodic.

\subsubsection{An example for $r < \infty$}\label{r<infty}

For $r < \infty$, let $\m_u$ denote the uniform Markov measure on $\F_r$ (the nonbacktracking simple symmetric random walk on $\F_r$). That is, the initial distribution $\pi$ is uniform (constant $\frac{1}{2r}$), and the transition matrix $P$ is defined by setting $P(a,a^{-1}) \defeq 0$ and $P(a,b) \defeq \frac{1}{2r-1}$ for all $a,b \in \Symb_r$ with $b \ne a^{-1}$. It is easy to check that $\pi$ is a stationary distribution for $P$ and $P$ is irreducible, hence $\m_u$ satisfies the hypothesis of \cref{forward_erg:boundary} and the shift map $\shift$ is $\P_{\m_u}$-ergodic.

\subsubsection{An example for $r = \infty$}

Recall that $\F_\infty = \gen{b_i}_{i < \infty}$ and $\Symb_\infty = \set{a_i}_{i<\infty}$, where $a_{2i} \defeq b_i$ and $a_{2i+1} \defeq b_i^{-1}$ for each $i < \infty$. Define an $\Symb_\infty \times \Symb_\infty$ matrix $P$ by setting its $a_i^\text{th}$ row to be the sequence $(\frac{1}{2^{j+1}})_{j < \infty}$ with a zero inserted for the entry corresponding to $a_i^{-1}$. More precisely, for all $i,j < \infty$,
\[
P(a_{2i},a_j)\defeq\begin{cases}
\frac{1}{2^{j+1}} &\text{if } j<2i+1
\\
0 &\text{if } j=2i+1
\\
\frac{1}{2^j} &\text{if }j>2i+1
\end{cases}
\;\text{ and }\;
P(a_{2i+1},a_j)\defeq\begin{cases}
\frac{1}{2^{j+1}} &\text{if }j<2i
\\
0 &\text{if }j=2i
\\
\frac{1}{2^j} &\text{if }j>2i.
\end{cases}
\]
Again, $P$ is irreducible, and one can check that every $a \in \Symb_\infty$ is a positive recurrent state (see \cite[paragraph above Theorem 5.5.12]{Dur} for the definition), so $P$ admits a positive stationary distribution $\pi$ by \cite[Theorems 5.5.11 and 5.5.12]{Dur}. Thus, the Markov measure $\m$ on $\Symb_\infty^{<\N}$ with transition matrix $P$ and initial distribution $\pi$ satisfies the hypothesis of \cref{forward_erg:boundary} and the shift map $\shift$ is $\P_\m$-ergodic.


\section{Application to pmp actions of free groups}\label{sec:pmp_actions}

Let $(X,\mu)$ be a standard probability space, and let $\F_r$ be the free group on $r$ generators, where $2 \le r < \infty$. 
As in \cref{subsec:boundary_actions}, let $\Symb_r \defeq \set{a_i}_{i<2r}$ be the standard symmetric set of generators and let $\partial\F_r$ be the boundary of $\F_r$ (i.e.\ the set of all infinite reduced words in $\Symb_r$).
Let $\m_u$ be the uniform Markov measure as in \cref{r<infty} whose initial distribution $\pi$ is the constant $\frac{1}{2r}$ vector and whose transition matrix $P$ is such that for all $i,j<r$, $P(a_i,a_j)=\frac{1}{2r-1}$ if $a_i\neq a_j^{-1}$, and $P(a_i,a_i^{-1})=0$ otherwise. 
In particular, for a word $w\in\F_r$ of length $n\geq 1$, $\m_u(w) = \frac{1}{(2r)(2r-1)^{n-1}}$.

\cref{intro:forward_erg:trees} is stated for $\m_u$ but we prove it here more generally for all stationary Markov measures $\m$ on $\F_r$ so long as the boundary action $\F_r \actson^\beta (\partial \F_r, \P_\m)$ is \dfn{weakly mixing}, i.e., the product of $\beta$ with any ergodic pmp action is weakly mixing.

Due to Kaimanovich, and Glasner and Weiss, this holds for the uniform Markov measure $\m_u$, i.e.\ the boundary action of $\F_r$ on $(\partial \F_r,\P_{\m_u})$ is weakly mixing.
Indeed, the measure $\P_{\m_u}$ coincides with the Poisson measure of the simple symmetric random walk on $\F_r$, so by \cite[Corollary following Theorem 2.4.6]{Kaimanovich} the boundary action of $\F_r$ (for $r \ge 2$) on $(\partial \F_r, \P_{\m_u})$ is doubly ergodic. 
This implies that it is weakly mixing, by \cite[Theorem 1.1]{GW}.
See also \cite[Example 5.1]{GW}.

\begin{remark}
    The upcoming work of the authors is devoted to characterizing all Markov measures $\m$ on $\F_r$ that make the boundary action of $\F_r$ on $(\partial \F_r,\P_{\m})$ weakly mixing.
\end{remark}

Recall that $\@T_{\Symb_r}^\bullet$ denotes the set of all finite subtrees of the (left) Cayley graph of $\F_r$ containing the identity.

\begin{theorem}\label{forward_erg:trees}
Let $2\le r < \infty$ and let $\F_r \actson^\alpha (X,\mu)$ be a (not necessarily free) pmp action of $\F_r$. 
Let $\m$ be a stationary Markov measure on $\Symb_r^{<\N}$ whose support is contained in $\F_r$. 
Suppose that the boundary action $\F_r \actson^\beta (\partial \F_r, \P_\m)$ is weakly mixing.
Then for every $f \in L^1(X,\mu)$, for $\mu$-a.e.\ $x \in X$,
\[
\frac{1}{\m(\tree)}\sum_{w \in \tree} f(w\cdot x) \m(w) \to \finv(x) \;\text{ as }\; \m(\tree) \to \infty,
\]
where $\tree$ ranges over $\@T_{\Symb_r}^\bullet$ and $\finv$ is the conditional expectation of $f$ with respect to the $\sigma$-algebra of $\alpha$-invariant Borel sets.
\end{theorem}

\begin{remark}
Note that because we defined Markov measures $\m$ via a \textit{row} stochastic matrix $P$, we may think that $\m$ assigns weights to words in $\Symb_r^{<\N}$ from left to right (algorithmically), i.e., $\m(w \conc a) = \m(w) P(w_{n-1}, a)$ for each $w \in \Symb_r^n$, $a \in \Symb_r$, and $n \ge 1$.
On the other hand, in the left Cayley graph of $\F_r$, the words grow from right to left, i.e., the edges are between the words $w$ and $a \conc w$.
In light of this, the setup of \cref{forward_erg:trees} may seem unnatural.
However, because $\m$ is assumed to be \textit{stationary}, it can be also expressed by a right-to-left recursive formula $\m(a \conc w) = \tilde{P}(a, w_0) \m(w)$, where $\tilde{P}$ is a \textit{column} stochastic matrix defined by $\tilde{P}(a,b) \defeq \frac{\m(a)}{\m(b)} P(a,b)$, for which the column-vector $\pi \defeq (\m(a))_{a \in \Symb_r}$ is stationary, i.e., $\tilde{P} \pi = \pi$.
\end{remark}

The rest of this section is devoted to the proof of \cref{forward_erg:trees}.
Let $\pi$ and $P$ be the initial distribution (row-vector) and the (row stochastic) transition matrix of $\m$, respectively. 

We will obtain \cref{forward_erg:trees} by applying \cref{backward_erg:trees} to the transformation (so-called \dfn{backward system})
$T : X \times \partial\F_r \to X \times \partial\F_r$, defined by 
\[
(x,y) \mapsto (y(0)^{-1} \cdot x, \shift(y))
\]
where $\shift : \partial\F_r \to \partial\F_r$ is the shift map. We equip $X \times \partial\F_r$ with the measure $\nu \defeq \mu \times \P_\m$.
Intuitively, we direct the standard Cayley graph of $\F_r$ towards every end (or rather its inverse) and for each such directing, consider the shift map towards that end.
This transforms every tree in the Cayley graph into a union of $2r$ backward trees of the shift map, and each of these trees either is negligible or has the correct average.

For $i_0, \ldots, i_n < 2r$, we will also abuse notation and write $[i_0, \ldots, i_n]$ for $[a_{i_0}, \ldots, a_{i_n}]$, $\pi(i_0)$ for $\pi(a_{i_0})$, and $P(i_0,i_1)$ for $P(a_{i_0},a_{i_1})$.

\begin{lemma}
$T$ is measure-preserving.
\end{lemma}

\begin{proof}
Let $U\subseteq X$ be Borel and $i_1, \ldots, i_n<2r$. Then for any $x \in X$, $i_0 < 2r$ and $y \in [i_0]$,
\begin{align*}
    (x,y) \in T^{-1}(U\times [i_1, \ldots, i_n]) 
    &\iff 
    (a_{i_0}^{-1}\cdot x,\shift(y))\in U \times [i_1, \ldots, i_n]
    \\
    &\iff
    (x,y) \in (a_{i_0} \cdot U) \times [i_0,i_1, \ldots, i_n],
\end{align*}
so $T^{-1}(U \times [i_1, \ldots, i_n]) = \bigsqcup_{i_0 < 2r} (a_{i_0} \cdot U) \times [i_0,i_1, \ldots, i_n]$. Hence,
\begin{align*}
    \nu(T^{-1}(U \times [i_1, \ldots, i_n])
    &= 
    \sum_{i_0<2r} \nu((a_{i_0} \cdot U) \times [i_0,i_1, \ldots, i_n])
    \\
    &=
    \mu(U) \cdot \sum_{i_0<2r} \P_\m([i_0, i_1, \ldots, i_n])
    \\
    &= 
    \mu(U) \cdot \P_\m(\shift^{-1}[i_1, \ldots, i_n])
    \\
    \eqcomment{$\m$ is stationary, so $\shift$ is pmp}
    &= 
    \nu(U \times [i_1, \ldots, i_n])
    .\qedhere
\end{align*}
\end{proof}

\begin{lemma}\label{same_equivalence_relation}
Let $\F_r \actson^\beta \partial \F_r$ be the boundary action of $\F_r$ as in \cref{subsec:boundary_actions}. Let $\F_r\actson^{\alpha\times\beta} X\times \partial \F_r$ be the diagonal action $g\cdot (x,y)=(g \cdot^\alpha x, g \cdot^\beta y)$.
Then the induced equivalence relation $E_{\alpha\times\beta}$ is the same as $E_T$ (the equivalence relation induced by $T$). 
\end{lemma}
\begin{proof}
Fix $(x,y)\in X \times \partial \F_r$. Then for $a \in \Symb_r$, if $a \neq y(0)^{-1}$, then $T(a \cdot (x,y)) = T(a \cdot x, a \conc y) = (x,y)$. If $a=y(0)^{-1}$, then $T(x,y)=(a\cdot x,\shift(y))= a \cdot (x,y)$.
On the other hand, $T(x,y)=(y(0)^{-1}x,\shift(y)) = y(0)^{-1} \cdot (x,y)$.
\end{proof}

Given that $E_T$ is the orbit equivalence relation of the diagonal action and the action of $\F_r$ on $X$ is pmp, the following lemma is clear, but we state it with an explicit formula for the Radon--Nikodym cocycle.

\begin{lemma}
$E_T$ is null-preserving on the conull set $X \times Y$, where $Y$ is the set of all infinite words $y \in \Symb_r$ in which each pair $y(n), y(n+1)$ of consequent letters has $P(y(n), y(n+1)) > 0$.
The Radon--Nikodym cocycle of $E_T \rest{X \times Y}$ with respect to $\nu$ only depends on the $Y$-coordinates and is equal to the Radon--Nikodym derivative of $E_\shift \rest{Y}$ with respect to $\P_\m$; more explicitly (by \cref{correctcocycle}):
\[
\coc_{T^n(x,y)}((x,y)) = \coc_{\shift^n(y)}(y) = \frac{\m \big(y(0) \conc y(1) \conc \dots \conc y(n)\big)}{\m \big(y(n)\big)}.
\]
\end{lemma}

\begin{proof}
By the cocycle identity, it is enough to prove
\[
\coc_{T(x,y)}((x,y)) = \frac{\pi\big(y(0)\big)}{\pi\big(y(1)\big)} P\big(y(0), y(1)\big).
\]
For this, it suffices to check that for Borel sets $U\subseteq X$ and $[i_0, \ldots, i_n]\subseteq \partial \F_r$, 
\[
\nu(T(U\times [i_0, \ldots, i_n]))=\int_{U\times [i_0, \ldots, i_n]} \frac{\pi(i_1)}{\pi(i_0)P(i_0,i_1)} \;  d\nu.
\]
To see this, observe that 
\begin{align*}
    \int_{U\times [i_0, \ldots, i_n]} \frac{\pi(i_1)}{\pi(i_0)P(i_0,i_1)} \;  d\nu 
    &= 
    \frac{\pi(i_1)}{\pi(i_0)P(i_0,i_1)}\cdot\nu(U\times [i_0, \ldots, i_n])
    \\
    &= 
    \frac{\pi(i_1)}{\pi(i_0)P(i_0,i_1)} \cdot \mu(U) \cdot \frac{\pi(i_0)P(i_0,i_1)}{\pi(i_1)} \cdot \P[i_1, \ldots, i_n]
    \\
    &= 
    \mu(U)\cdot\P_\m[i_1, \ldots, i_n]
    \\
    &= 
    \mu(a_{i_0} \cdot U) \cdot \P_\m[i_1, \ldots, i_n]
    \\
    &= 
    \nu(T(U\times [i_0, \ldots, i_n]))
    .\qedhere
\end{align*}
\end{proof}

\begin{lemma}\label{X-ergodic=>XxY-ergodic}
If $\F_r\actson^\alpha (X,\mu)$ is ergodic, then $T$ is ergodic.
\end{lemma}

\begin{proof}
By \cref{same_equivalence_relation}, it is enough to show $E_{\alpha\times\beta}$ is ergodic, but this follows from the hypothesis that the boundary action of $\F_r$ on $(\partial \F_r,\P_{\m})$ is weakly mixing.
\end{proof}

\begin{lemma}\label{conditional_exp_coincide}
For $f \in L^1(X,\mu)$, define $F \in L^1(X \times \partial \F_r, \mu\times \P_\m)$ by $F(x,y) \defeq f(x)$. Then for $\mu$-a.e.\ $x \in X$ and $\#P_\m$-a.e.\ $y \in \partial \F_r$, $\overline{F}(x,y) = \finv(x)$, where $\finv$ and $\overline{F}$ are the conditional expectations of $f$ and $F$ with respect to the $\sigma$-algebras of $\alpha$ and $T$-invariant Borel sets, respectively.
\end{lemma}

\begin{proof}
If the action $\F_r \actson^\alpha (X,\mu)$ is ergodic, then so is $T$ by \cref{X-ergodic=>XxY-ergodic}, and hence $\finv[F](x,y)= \int_{X \times \partial \F_r} F \; d \nu = \int_X f(x) \; d\mu(x) = \finv(x)$. 

In the general case, we use the ergodic decomposition theorem for pmp countable Borel equivalence relations (see \cref{defn:ergodic_dec} and the following paragraph), which gives us an $E_\alpha$-ergodic decomposition $\epsilon : X \to P(X)$. By \cref{erg_dec--cond_exp}, $\finv(x) = \int_X f \;d\epsilon_x$ for $\mu$-a.e.\ $x \in X$. 

But $\int_X f \;d\epsilon_x = \int_{X \times \partial \F_r} F \;d (\epsilon_x \times \#P_\m)$ by the definition of $F$, so it remains to show that $\finv[F](x,y) = \int_{X \times \partial \F_r} F \;d (\epsilon_x \times \#P_m)$ for a.e.\ $x \in X$ and $\#P_\m$-a.e.\ $y \in \partial \F_r$.

To this end, for any $H \in L^1(X \times \partial \F_r, \mu \times \#P)$, Fubini's theorem gives
\begin{align*}
\int_{X \times \partial \F_r} H \;d (\mu \times \#P_\m) 
&= 
\int_X \int_{\partial \F_r} H(x,y) \;d\#P_\m(y) \;d\mu(x)
\\
\eqcomment{\cref{erg_dec--cond_exp}}
&= 
\int_X \int_X \int_{\partial \F_r} H(z,y) \;d\#P_\m(y) \;d\epsilon_x(z) \;d\mu(x)
\\
\eqcomment{Fubini}
&=
\int_X \int_{X \times \partial \F_r} H(z,y) \;d(\epsilon_x \times \#P_\m)(z,y) \;d\mu(x)
\\
\eqcomment{dummy integration}
&=
\int_{X \times \partial \F_r} \int_{X \times \partial \F_r} H(z,y) \;d(\epsilon_x \times \#P_\m)(z,y) \;d(\mu \times \#P_\m)(x,y').
\end{align*}
Recalling, in addition, that each measure $\epsilon_x \times \#P_\m$ is $T$-ergodic, by \cref{X-ergodic=>XxY-ergodic}, we see that the map $(x,y) \mapsto \epsilon_x \times \#P_\m$ is the $E_T$-ergodic decomposition of $\mu \times \#P_\m$ (\cref{defn:ergodic_dec}). Thus, again by \cref{erg_dec--cond_exp}, $\finv[F](x,y) = \int_{X \times \partial \F_r} F \;d(\epsilon_x \times \#P_\m)$ for $(\mu \times \#P_\m)$-a.e.\ $(x,y) \in X \times \partial \F_r$, finishing the proof.
\end{proof}

We now show that \cref{backward_erg:trees} applied to $T$ yields \cref{forward_erg:trees}:

\begin{proof}[Proof of \cref{forward_erg:trees}]
For each $a \in S_r$, define the partial function $\^a : X \times \partial \F_r \partialto X \times \partial \F_r$ by $\^a(x,y) \defeq (a \cdot x, a \conc y)$ for $(x,y) \in \dom(\^a) \defeq X \times (\partial \F_r \setminus \bigcup_{\substack{b \in \Symb_r \\ P(a,b) = 0}} [b])$.
Let $\^\Symb_r \defeq \set{\^a : a \in \Symb_r}$, and notice that for each $(x,y) \in X \times \partial \F_r$,
\[
T^{-1}(x,y) = \set{(a_i \cdot x, a_i \conc y) : i < 2r \text{ and } a_i\neq y(0)^{-1}} = \^\Symb_r \cdot (x,y),
\]
where $\^\Symb_r \cdot (x,y)$ is defined as in \cref{subsec:right-inverses}.
Thus, $\^\Symb_r$ is a complete set of Borel partial right inverses of $T$.
For a reduced word $w \defeq s_1 s_2 \dots s_n$ with $s_i \in \Symb_r$, we let $\^w \defeq \^{s_1} \circ \^{s_2} \circ \dots \circ \^{s_n}$ be the induced partial function on $X \times \partial \F_r$.
Lastly, we put $\^\tree \defeq \set{\^w : w \in \tree}$ for each $\tree \in \@T_{\Symb_r}^\bullet$.

Fix $f \in L^1(X,\mu)$. 
Define $F \in L^1(X \times \partial \F_r, \nu)$ by $F(x,y) \defeq f(x)$. 
For $\tree \in \@T_{\Symb_r}^\bullet$, and $(x,y) \in X \times \partial \F_r$, we will abuse notation and write $\coc_{y_2}(y_1)$ for $\coc_{(x_2,y_2)}((x_1,y_1))$, as well as $\coc_y(\^\tree \cdot y)$ for $\coc_{(x,y)}(\^\tree \cdot (x,y))$, since $\coc$ does not depend on $x$. 

\begin{claim}\label{weighted_sums_over_tree}
For each $\tree \in \@T_{\Symb_r}^\bullet$, $x \in X$, and $y_i \in [i]$ where $i$ ranges in $\set{0, \dots, 2r-1}$,
\begin{enumerate}[(a)]
\item\label{weight_of_tree} $\m(\tree) = \sum_{i<2r} \pi(i) \coc_{y_i}(\^\tree \cdot y_i)$.

\item\label{weighted_sum} $
\sum_{w \in \tree} f(w \cdot x) \m(w) 
= 
\sum_{i < 2r} A_F^\coc[\^\tree \cdot (x,y_i)] \pi(i) \coc_{y_i}(\^\tree \cdot y_i).
$
\end{enumerate}
\end{claim}

\begin{pf}
Part \labelcref{weight_of_tree} follows from \labelcref{weighted_sum} by plugging-in $f \defeq 1$.
As for \labelcref{weighted_sum}, we compute: 
\begin{align*}
\sum_{i < 2r} \pi(i) \coc_{y_i}(\^\tree \cdot y_i) A_F^\coc[\^\tree \cdot (x,y_i)]
&= 
\sum_{i < 2r} \pi(i) \sum_{\substack{w \in \tree, \\ P(w_{|w|-1}, i)  > 0}} f(w \cdot x) \coc_{y_i}(w \conc y_i) 
\\
\eqcomment{$m(w \conc i) = 0$ when $P(w_{|w|-1}, i) =0$}
&= 
\sum_{i<2r} \pi(i) \sum_{w \in \tree} f(w \cdot x) \frac{\m(w \conc i)}{\pi(i)}
\\
\eqcomment{Fubini}
&= 
\sum_{w \in \tree}f(w\cdot x)\sum_{i<2r}\m(w \conc i)
\\
&= 
\sum_{w \in \tree} f(w \cdot x) \m(w)
.\qedhere
\end{align*}
\end{pf}

Applying \cref{backward_erg:trees,maximal_ergodic} to the transformation $T$, we get that the conclusions of these theorems hold for $\nu$-a.e.\ $(x,y)\in X\times  \partial \F_r$.
Thus, for each point $x$ in a $\mu$-conull set $X' \subseteq X$ and for each $i < 2r$, there is $y_i \in  \partial \F_r \cap [i]$ such that the conclusions of \cref{backward_erg:trees,maximal_ergodic} hold at $(x, y_i)$. 
Fix $x \in X'$ and $(y_i)_{i < 2r}$ as above, as well as $\e > 0$.
Then the choices of $x$ and $(y_i)_{i < 2r}$ (namely, \cref{backward_erg:trees,maximal_ergodic}), and the fact that $r < \infty$, yield an $M > 0$ large enough so that for each $i<2r$ and $\tree_{(x,y_i)} \in \@T_{(x,y_i)}$:
\begin{enumerate}[(i)]
    \item $|\finv(x)| \le M$ and $|A_f^\coc[\tree_{(x,y_i)}]| \le M$,

    \item $A_f^\coc[\tree_{(x,y_i)}] \approx_{\pi(i) \e} \overline{F}(x,y_i)$ whenever $\coc_{y_i}(\tree_{(x,y_i)}) \ge M$.
\end{enumerate}
It remains to show $\frac{1}{\m(\tree)} \sum_{w \in \tree}  f(w \cdot x) \m(w) \approx_\e \finv(x)$ for each $\tree \in \@T_{\Symb_r}^\bullet$ with $\m(\tree) > \frac{2 M^2}{\e}$.

\begin{claim}\label{average_vs_cond-exp}
$A_F^\coc[\^\tree \cdot (x,y_i)] \frac{\pi(i) \coc_{y_i}(\^\tree \cdot y_i)}{\m(\tree)} \approx_{\pi(i) \e} \overline{F}(x, y_i) \frac{\pi(i) \coc_{y_i}(\^\tree \cdot y_i)}{\m(\tree)}$ for each $i < 2r$.   
\end{claim}

\begin{pf}
For each $i < 2r$, there are two cases.

If $\coc_{y_i}(\^\tree \cdot y_i) < M$, then $\frac{\pi(i) \coc_{y_i}(\^\tree \cdot y_i)}{\m(\tree)} < \frac{\pi(i) M \e}{2 M^2} = \frac{\pi(i) \e}{2 M}$, and both $|A_F^\coc[\^\tree \cdot (x,y_i)]|$ and $|\overline{F}(x, y_i)|$ are at most $M$, so both quantities in question are less than $\frac{\pi(i) \e}{2}$ and hence within $\pi(i) \e$ of each other.

If $\coc_{y_i}(\^\tree \cdot y_i) \ge M$, then $A_F^\coc[\^\tree \cdot (x,y_i)] \approx_{\pi(i) \e} \overline{F}(x,y_i)$, so we are done because $\frac{\pi(i) \coc_{y_i}(\^\tree \cdot y_i)}{\m(\tree)} \le 1$ by \cref{weighted_sums_over_tree}\labelcref{weight_of_tree}.
\end{pf}

It remains to compute, using \cref{weighted_sums_over_tree}\labelcref{weighted_sum} on the first line:
\begin{align*}
\frac{1}{\m(\tree)} \sum_{w \in \tree}  f(w \cdot x) \m(w)
& =
\sum_{i < 2r} A_F^\coc[\^\tree \cdot (x,y_i)] \frac{\pi(i) \coc_{y_i}(\^\tree \cdot y_i)}{\m(\tree)}
\\
\eqcomment{\cref{average_vs_cond-exp}}
& \approx_\e
\sum_{i<2r} \overline{F}(x, y_i) \frac{\pi(i) \coc_{y_i}(\^\tree \cdot y_i)}{\m(\tree)} 
\\
\eqcomment{\cref{conditional_exp_coincide}}
& = 
\finv(x) \sum_{i<2r} \frac{\pi(i) \coc_{y_i}(\^\tree \cdot y_i)}{\m(\tree)}
\\
\eqcomment{\cref{weighted_sums_over_tree}\labelcref{weight_of_tree}}
& =
\finv(x) 
.\qedhere
\end{align*}
\end{proof}


\def\MR#1{}
\bibliographystyle{amsalpha} 
\bibliography{ref} 


\end{document}